\documentclass[leqno]{amsart}
\usepackage{amssymb}
\usepackage{amsmath, amsfonts, vmargin, enumerate}
\usepackage{dsfont}
\input xypic

\newtheorem{thm}{Theorem}[section]
\newtheorem{prop}[thm]{Proposition}
\newtheorem{cor}[thm]{Corollary}
\newtheorem{lem}[thm]{Lemma}
\newtheorem{defi}[thm]{Definition}
\newtheorem{remark}[thm]{Remark}
\newtheorem{example}[thm]{Example}
\newtheorem{pb}[thm]{Problem}

\newenvironment{rk}{\begin{remark}\rm}{\end{remark}}

\numberwithin{equation}{section}

\newcommand{\real}{{\mathbb R}}
\newcommand{\nat}{{\mathbb N}}
\newcommand{\ent}{{\mathbb Z}}

\newcommand{\un}{{\mathds {1}}}

\newcommand{\norm}[1]{\left\Vert#1\right\Vert}
\newcommand{\abs}[1]{\left\vert#1\right\vert}

\newcommand{\A}{{\mathcal A}}

\newcommand{\C}{{\mathcal C}}

\renewcommand{\H}{{\mathcal H}}
\newcommand{\I}{{\mathcal I}}
\newcommand{\K}{{\mathcal K}}

\newcommand{\V}{\mathcal{V}}

\renewcommand{\a}{\alpha}
\renewcommand{\b}{\beta}

\renewcommand{\d}{\delta}

\newcommand{\f}{\varphi}

\renewcommand{\l}{\lambda}
\renewcommand{\O}{\Omega}

\newcommand{\s}{\sigma}

\newcommand{\8}{\infty}
\newcommand{\el}{\ell}

\newcommand{\wt}{\widetilde}

\newcommand{\n}{\noindent}

\newcommand{\be}{\begin{eqnarray*}}
\newcommand{\ee}{\end{eqnarray*}}
\newcommand{\beq}{\begin{equation}}
\newcommand{\eeq}{\end{equation}}
\newcommand{\beqn}{\begin{equation*}}
\newcommand{\eeqn}{\end{equation*}}
\newcommand{\cqd}{\hfill$\Box$}

\begin{document}

\title[Weighted variation inequalities]{Weighted variation inequalities for differential operators and singular integrals}


\author{Tao Ma}
\address{School of Mathematics and Statistics, Wuhan University, Wuhan 430072, China }
\email{tma.math@whu.edu.cn}

\author{Jos\'e Luis Torrea}
\address{Departamento de Matem\'aticas and ICMAT-CSIC-UAM-UCM-UC3M, Facultad de Ciencias, Universidad
Aut\'onoma de Madrid, 28049 Madrid, Spain }
\email{joseluis.torrea@uam.es}

\author{Quanhua Xu}
\address{School of Mathematics and Statistics, Wuhan University, Wuhan 430072, China and Laboratoire de Math{\'e}matiques, Universit{\'e} de Franche-Comt{\'e},
25030 Besan\c{c}on Cedex, France}
\email{qxu@univ-fcomte.fr}

\date{}
\maketitle

\begin{abstract}
 We prove weighted strong $q$-variation inequalities with $2<q<\8$ for differential and singular integral operators. For the first family of operators the weights used can be either Sawyer's one-sided $A^+_p$ weights or Muckenhoupt's  $A_p$ weights according to that the differential operators in consideration are one-sided or symmetric. We use only Muckenhoupt's  $A_p$ weights for the second family. All these inequalities hold equally in the vector-valued case, that is, for functions with values in $\el^\rho$ for $1<\rho<\8$.  As application, we show variation inequalities for  mean bounded positive invertible operators on $L^p$ with positive inverses.
\end{abstract}

\bigskip



\section{Introduction}


Variation inequalities have been the subject of many recent research papers  in probability, ergodic theory and harmonic analysis.  One important feature of these inequalities is the fact that they immediately imply the pointwise convergence of the underlying family of operators without using the Banach principle via the corresponding maximal inequality. Moreover, these variation inequalities  can be used to measure the speed of convergence of the family.

The first variation inequality was proved by L\'epingle \cite{Le} for martingales which  improves the classical Doob maximal inequality (see also \cite{PX} for a different approach and related results). Thirteen years later, Bourgain \cite{Bour} proved the variation inequality  for the ergodic averages of a dynamic system. Bourgain's work has inaugurated a new research direction in ergodic theory and harmonic analysis. It was considerably improved by subsequent works and largely extended to many other operators in ergodic theory; see, for instance, \cite{jrw-high, kaufman, LMX}. Almost in the same period, variation inequalities have been studied in harmonic analysis too. The first work on this subject is \cite{CJRW1} in which Campbell, Jones,  Reinhold and  Wierdl proved the variation inequalities for the Hilbert transform. Since then many other publications came to enrich the literature on this subject (cf. e.g., \cite{CJRW2, DMC,  GT, HLP, jsw, JW, Mas, MTo, MTo2, OSTTW}).

\medskip

The purpose of this paper is to study weighted variation inequalities for differential and singular integral operators. The first family of operators can be considered both in the discrete and continuous cases. To fix ideas let us confine ourselves to the former. Given a function $f$ on $\ent$ define
 $$A^+_N(f)(n)=\frac{1}{N+1}\sum_{i=0}^{N}f(n+i)$$
and $\A^+ (f)(n)=\{A^+_N(f)(n)\}_{N\ge0}$. $\A^+$ is an operator mapping functions on $\ent$ to sequences of functions on $\ent$. We will study the variation of the sequence $\A^+ (f)(n)$.

Let $1\le q<\8$ and $a=\{a_N\}_{N\ge0}$ be a sequence of complex numbers. The $q$-variation of $a$ is defined as
 \beq\label{variation}
 \|a\|_{v_q} =\sup\big(\sum^\infty_{j=0}|a_{N_j}-a_{N_{j+1}}|^q\big)^{1/q},
 \eeq
where the supremum runs over all increasing sequences $\{N_j\}$ of nonnegative integers. Let $v_q$ denote the space of all sequences with finite $q$-variation. This is a Banach space modulo constant functions.
Let $ \V_q\A^+ (f)(n)=\|\A^+(f)(n)\|_{v_q}$. Thus   the operator $\V_q\A^+$ sends  functions on $\ent$ to  nonnegative functions on $\ent$. Throughout the paper,  $\V_q$ designates the operator which maps a sequence to its $q$-variation. Later in the continuous case, the same symbol $\V_q$ will also be the operator mapping functions on $(0,\,\8)$ to their $q$-variations.

Bourgain's theorem quoted before asserts that for any $2<q<\8$, $\V_q\A^+$ is bounded on $\el^2(\ent)$. This result was extended to $\el^p(\ent)$ for any $1<p<\8$ in \cite{kaufman}. Moreover. Jones {\it et al} also proved that $\V_q\A^+$ is of weak type $(1, 1)$, namely, it maps $\el^1(\ent)$ into $\el^{1,\8}(\ent)$.

\medskip

These $q$-variation inequalities improve the classical (one-sided) Hardy-Littlewood maximal inequality that we recall as follows. Let
 $$M^+(f)(n) = \sup_{N\ge0} A^+_N(|f|)(n).$$
Then $M^+$ is of type $(p, p)$ for $1<p\le\8$ and weak type $(1, 1)$. This follows from the previous $q$-variation results by virtue of the trivial inequality $M^+(f)\le \V_q\A^+(f) +f(0)$ for any nonnegative function $f$.

Sawyer \cite{Sawyer} characterized  the weights $w$ on $\ent$ for which $M^+$ is bounded on $\el^p(\ent, w)$ with $1<p<\8$, and maps $\el^1(\ent, w)$ into $\el^{1,\8}(\ent, w)$. These are the so-called $A^+_p$ weights that are defined below.
Let $w$ be a positive  function on  $\ent$.
\begin{enumerate}[$\bullet$]
\item   $w\in A_1^+$ if there exists a constant $C$ such that
 $$\sum_{n-k}^n w(i) \le C (k+1)\min\{w(i) : i \in [n,\,n+k]\},\quad\forall \; n\in\ent,\; k\ge0.$$

\item  $w\in A_p^+$ (with $1<p<\8$) if there exists a constant $C$ such that
 $$\sum_{i=0}^k w(n+i) \,\big(\sum_{i=k}^{2k} w(n+i)^{-\frac{1}{p-1}}\big)^{p-1} \le C(k+1)^p,\quad\forall \; n\in\ent,\; k\ge0.$$
\end{enumerate}

It is thus natural to wonder whether Sawyer's weighted inequalities hold for $\V_q\A^+$ in place of $M^+$. The first main result of our paper provides an  affirmative answer to this question. Namely, $\V_q\A^+$ is bounded on $\el^p(\ent, w)$ for $1<p<\8$ and $w\in A^+_p$, and from $\el^1(\ent, w)$ into $\el^{1,\8}(\ent, w)$ for $w\in A^+_1$.

\medskip

Sawyer's result is the one-sided analogue of Muckenhoupt's celebrated characterization of  $A_p$ weights for the symmetric Hardy-Littlewood maximal function. Our reference for real variable harmonic analysis is \cite{gar-rubio}. The reader is also referred to this book for all results quoted below but without reference. Let us just recall the definition of $A_p$ weights.  For a nonnegative function $w$ on $\ent$, by definition
\begin{enumerate}[$\bullet$]
 \item  $w\in A_p$ (with $1<p<\8$) if there exists a constant $C$ such that
 $$\sum_{i\in I} w(i) \,\big(\sum_{i\in I} w(i)^{-\frac{1}{p-1}}\big)^{p-1} \le C|I|^p$$
for any interval $I\subset\ent$;
 \item   $w\in A_1$ if there exists a constant $C$ such that
 $$M(w)\le Cw.$$
 \end{enumerate}
Here for a function $f$ on $\ent$, $M(f)$ denotes the usual Hardy-Littlewood maximal function:
 $$M(f)(n)=\sup_I\frac1{|I|}\sum_{i\in I} |f(i)|,$$
where  the supremum runs over all intervals  containing $n$.

On the other hand, it is well known that  $A_p$ weights  can be also characterized by the boundedness of the Hilbert transform $H$. This time,  it is more convenient to work  on $\real$ instead of $\ent$. The above definition of $A_p$ weights remains valid in $\real$ without any change.

We make a convention at this occasion: we will use the same notational system for $\ent$ and $\real$.  For example, $M(f)$ also denotes the  Hardy-Littlewood maximal function of $f$ on $\real$.

The Hilbert transform is the following singular integral (taking in the principal value sense):
 \beq\label{H}
 H(f)(x)=\int_{\real}\frac{f(y)}{x-y}\,dy.
 \eeq
Let $w$ be a weight on $\real$. Then $H$ is bounded on $L^p(\real, w)$ with $1<p<\8$ if and only if $w\in A_p$, and maps  $L^1(\real, w)$ to $L^{1,\8}(\real, w)$ if and only if $w\in A_1$. Because of the singularity of the integral above, it is more convenient to consider its truncations:
   \beq\label{Ht}
   H_t(f)(x)=\int_{|x-y|>t}\frac{f(y)}{x-y}\,dy.
   \eeq
 Let $H^*(f)(x)=\sup_{t>0}\big|H_t(f)(x)\big|$. Then the above statement still holds with $H^*$ instead of $H$.

\medskip

In the spirit of the weighted variation inequality for the differential operators, we wish to show the $q$-variation analogue of the last statement, i.e., replacing the maximal function $H^*(f)$ by the corresponding $q$-variation. The $q$-variation of a family indexed by a continuous time $t$ is defined exactly as in \eqref{variation}. More precisely, for a family  $a=\{a_t\}_{t>0}$ of complex numbers we define
 $$\|a\|_{v_q} =\sup\big(\sum^\infty_{j=0}|a_{t_j}-a_{t_{j+1}}|^q\big)^{1/q},$$
where the supremum runs over all increasing sequences $\{t_j\}$ of positive numbers. We use again $v_q$ to denote the space of all functions on $(0,\,\8)$ with finite $q$-variation.

Then let  $\H (f)(x)=\{H_t(f)(x)\}_{t>0}$ and $\V_q\H(f)(x)=\|\H (f)(x)\|_{v_q}$. A special case of our second main theorem asserts that for $2<q<\8$ the operator $\V_q\H$ is bounded on $L^p(\real, w)$ for $1<p<\8$ and $w\in A_p$, and from $L^1(\real, w)$ into $L^{1,\8}(\real, w)$ for $w\in A_1$. This is the weighted version of the main result of \cite{CJRW1}. In fact, we show weighted $q$-variation inequalities for a singular integral with a regular kernel provided that the associated $q$-variation operator is bounded on $L^p(\real)$ for some $1<p<\8$. 

\medskip

In the literature, along with variation inequalities, another family of inequalities have equally received much attention. They are oscillation inequalities. Given a fixed sequence $\{N_j\}$ of nonnegative integers,   the oscillation of a sequence $a=\{a_N\}_{N\ge0}$ with respect to $\{N_j\}$ is defined as
 $$\mathcal{O}(a) = \big(\sum_{j=0}^\infty \sup_{N_{j} \leq N< M< N_{j+1}}|a_N-a_M|^2 \big)^{1/2}.$$
Almost all results in this paper are valid equally for oscillation with similar arguments. We leave this part to the interested reader.

\medskip

The paper is organized as follows. In the next section we prove the weighted $q$-variation inequalities and their vector-valued analogues for the differential operators. Our proof of the type $(p,p)$ inequality uses some standard techniques in harmonic analysis. However, because the kernels of the differential operators are not regular, the proof of the weak type $(1,1)$ inequality requires a careful analysis of them. In section~\ref{Singular integrals}, the same weighted inequalities are proved for singular integral operators with regular kernels under the assumption that their associated $q$-variation operators are bounded on $L^p(\real)$ for some $1<p<\8$. The Hilbert transform and Cauchy integral on a Lipschitz curve are such singular integral operators.  The proof is similar to the previous one for differential operators. We would like to emphasize that our proofs of the weak type $(1,1)$ case for both families of operators are simpler than the existing proofs in  similar situations since most of them are divided into two parts by showing separately the corresponding inequalities for the short and long variations; see the proof of the unweighted weak type $(1, 1)$ for differential operators in \cite{kaufman}, and that for the Hilbert transform and singular integrals in \cite{CJRW1, CJRW2}. Section~\ref{sect-vector} is devoted to the vector-valued extension of the results in the preceding two ones. We show there that the previous results also hold for functions with values in $\el^\rho$ for $1<\rho<\8$. These vector-valued variation inequalities are new in the unweighted case too. The last section gives an application to ergodic theory for mean bounded positive invertible operators on $L^p$ with positive inverses.

\medskip

In a subsequent paper we will study higher dimensional  case. Most results of the present paper have higher dimensional analogues.  However, the arguments in the higher dimensional case are often more complicated and technical. On the other hand, we do not know how to extend Theorem~\ref{Thm:variation-Ap+} for the one-sided differential operators to higher dimensions.

\medskip

We  end this introduction by a convention: the symbol $A\lesssim B$ means an inequality up to a constant that may depend on the indices $p, q$, the weights $w$, the kernels $K$, etc. but never on the functions $f$ in consideration.


\section{Differentiable operators}


In this section we study weighted variation inequalities for differential operators. These operators can be defined  both in the discrete and continuous cases. The methods dealing the two  have no major differences. Thus we will focus our attention on the discrete case.   The following is the main result of this section.

\begin{thm}\label{Thm:variation-Ap+}
 Let $q>2$.
\begin{enumerate}[\rm (i)]
  \item Let $1<p<\infty$. The operator $\V_q\A^+$is bounded on $\ell^p(\ent, w)$  if and only if $w \in A^+_p$.
  \item The operator $\V_q\A^+$ maps $\ell^1(\ent, w)$ into $\ell^{1,\infty}(\ent, w)$ if and only if $w \in A^+_1$.
\end{enumerate}
 \end{thm}

By extrapolation, part (ii) implies part (i). But our proof of (ii) depends on (i). The proof of the theorem requires the following lemma from \cite{MT2}. Let  $ f^{+,\sharp}$ denote  the one-sided sharp maximal function of $f$:
$$ f^{+,\sharp} (n) = \sup_{k\in \mathbb{N}} \frac1{k+1} \sum_{i=n}^{n+k}\big(f(i)- \frac1{k+1} \sum_{j=n+k}^{n +2k}f(j) \big)^+.$$

\begin{lem}\label{Le:malaga}
 Let $1\le p< \infty$ and $w\in \bigcup_{p\ge1} A_p^+$. Then
 \begin{eqnarray*}
 \sum_{n\in\ent} \big(M^+(f)(n)\big)^p w(n) \lesssim \sum_{n\in\ent} |f^{+,\sharp}(n)|^p w(n)
 \end{eqnarray*}
whenever the left hand side is finite.
\end{lem}

The following elementary fact will be also used in the proof.

\begin{lem}\label{basic ineq}
 Let $r>1$ and $(t_j)$ be an increasing sequence of positive numbers. Then
 $$\sum_{j=0}^\8\frac{(t_{j+1}-t_j)^r}{t_{j+1}^rt_j^{r-1}} \lesssim \frac1{t_0^{r-1}}.$$
\end{lem}

\begin{proof}
 This inequality is easily checked as follows:
 \begin{align*}
 \sum_{j=0}^\8\frac{(t_{j+1}-t_j)^r}{t_{j+1}^rt_j^{r-1}}
 &\lesssim\sum_{j: t_{j+1}\ge 2t_j}\frac{1}{t_j^{r-1}}
 +\sum_{j: t_{j+1}< 2t_j}\frac{t_{j+1}-t_j}{t_{j+1}^r}\\
 &\lesssim\frac{1}{t_0^{r-1}}\sum_{j=0}^\8\frac{1}{2^j}
 +\int_{t_0}^\8\frac{dt}{t^r}\approx \frac{1}{t_0^{r-1}}.
 \end{align*}
 \end{proof}

The following variant for $M^+$ of the  classical Calder\'on-Zygmund decomposition will be crucial for the proof of the weak type $(1, 1)$ inequality in part (ii).

\begin{lem}\label{1-sided CZ}
 Let  $f$ be a finitely supported function on $\ent$ and  $\lambda >0$.  Let  $\O= \{n: M^+f(n) > \lambda\}$. Then $\O$ can be decomposed into finitely many disjoint intervals of integers: $\O = \bigcup_iI_i$ with the following properties
 \begin{enumerate}[$\bullet$]
 \item $|f(n)| \le \lambda$ for all $n\notin\O$;
 \item $\displaystyle |\O| \le \frac1{\lambda}\,\|f\|_1$;
 \item $\displaystyle \lambda < \frac1{|I_i|} \sum_{n\in I_i} |f(n)| \le 2 \lambda$.
 \end{enumerate}
 \end{lem}

\begin{proof}
 Recall that the classical Calder\'on-Zygmund decomposition uses the usual Hardy-Littlewood maximal function $M$ instead of the one-sided $M^+$. However, the standard proof, for instance, that of \cite[Theorem~II.1.2]{gar-rubio}, can be easily modified to the present situation. The only new fact needed is \cite[Lemma~2.1]{Sawyer}. We omit the details.
\end{proof}

\n\emph{Proof of Theorem~\ref{Thm:variation-Ap+}.} For the reason of presentation, we will denote $\V_q\A^+$ simply by $\V_q$  throughout this proof.

 (i)  The necessity is clear by the trivial inequality $M^+(f)\le \V_q(f)+f(0)$ for any $f\ge0$. For the converse direction, we will prove the following inequality
 \begin{equation}\label{Inequality:sharp}
 \left(\V_q(f)\right)^{+,\sharp} \lesssim M^+_r(f)
 \end{equation}
for any finitely supported function $f$ on $\ent$ and for  $r>1$ (sufficiently close to 1), where $M^+_r(f)=\big(M^+(|f|^r)\big)^{1/r}.$
Assuming \eqref{Inequality:sharp}, we easily conclude the sufficiency of (i).  Indeed, by \cite{Sawyer} there exists $r>1$ such that $w$ belongs to $A^+_{p/r}$ too. Thus by \cite{Sawyer} again
 $$ \sum_{n\in\ent} \left(M^+_r(f)(n)\right)^p w(n) \lesssim \sum_{n\in\ent} |f|^p w(n).$$
Then Lemma~\ref{Le:malaga}   and \eqref{Inequality:sharp} imply the desired sufficiency.

Let us prove \eqref{Inequality:sharp}.  Let $f$ be a finitely supported function  on $\ent$ and $n_0\in\ent$. Recall that
 $$\left(\V_q(f)\right)^{+,\sharp}(n_0)=
 \sup_{k\ge1} \frac{1}{k+1} \sum_{i=n_0}^{n_0+k} \Big(\V_q(f)(i)-\frac{1}{k+1}\sum_{j= n_0+k}^{n_0+2k} \V_q(f)(j)\Big)^+.$$
Fix  $k\ge1$. We decompose $f$ as $f=f_1+f_2+f_3$, where $f_1=f\un_{[n_0,\, n_0+3k]}$ and $f_2=f\un_{(n_0+3k,\,\infty)}$. Then
 \begin{align*}
 \frac{1}{k+1}
 &\sum_{i=n_0}^{n_0+k} \Big(\V_q(f)(i) - \frac{1}{k+1}\sum_{j= n_0+k}^{n_0+2k} \V_q(f)(j)\Big)^+\\
 &\leq \frac{1}{k+1} \sum_{n=n_0}^{n_0+k} \big|\mathcal{V}_{q}(f)(n) - \mathcal{V}_{q}(f_2)(n_0)\big|\\
 & \quad\quad +\frac1{k+1} \sum_{n=n_0+k}^{n_0+2k} \big|\mathcal{V}_{q}(f)(n) - \mathcal{V}_{q}(f_2)(n_0)\big|\\
 &\leq \frac{1}{k+1} \sum_{n=n_0}^{n_0+k} \big|\mathcal{V}_{q}(f)(n) - \mathcal{V}_{q}(f_2)(n_0)\big|\\
 & \quad\quad +\frac2{2k+1} \sum_{n=n_0}^{n_0+2k}\big|\mathcal{V}_{q}(f)(n) - \mathcal{V}_{q}(f_2)(n_0)\big|.
 \end{align*}
We only need to estimate the first part of the last sum, the second one being handled similarly (with $2k$ instead of $k$). Noting that $A_N(f_3)(n)=0$ for every $n\geq n_0$, we have
 \begin{align*}
 \frac1{k+1} \sum_{n=n_0}^{n_0+k} \big| \mathcal{V}_{q}(f)(n) - \mathcal{V}_{q}(f_2)(n_0) \big|
 &=\frac1{k+1} \sum_{n=n_0}^{n_0+k} \big| \|\mathcal{A}^+(f)(n)\|_{v_q}  - \|\mathcal{A}^+(f_2)(n_0)\|_{v_q} \big| \\
 &\le \frac1{k+1} \sum_{n=n_0}^{n_0+k}  \|\mathcal{A}^+(f)(n)  - \mathcal{A}^+(f_2)(n_0)\|_{v_q} \\
 &\le  \frac1{k+1} \sum_{n=n_0}^{n_0+k}  \|\mathcal{A}^+(f_1)(n)  \|_{v_q}\\
 &\quad+ \frac1{k+1} \sum_{n=n_0}^{n_0+k}  \|\mathcal{A}^+(f_2)(n)  - \mathcal{A}^+(f_2)(n_0)\|_{v_q} \\
 &\;{\mathop=^{\rm def}}\;  E_1+E_2 .
  \end{align*}
By the H\"older inequality and the $\el^r$-boundedness of $\V_q$ proved in \cite{kaufman}, we get
  \begin{align*}
  E_1
  & \le \big(\frac{1}{k+1} \sum_{n=n_0}^{n_0+k}  \|\mathcal{A}^+(f_1)(n)  \|^r_{v_q} \big)^{1/r}
  \lesssim  \big(\frac1{k+1} \sum_{n\in\, \ent}  |f_1(n)|^r \big)^{1/r} \\
  &=  \big(\frac1{k+1} \sum_{n=n_0}^{n_0+3k}|f(n)|^r)^{1/r} \lesssim M^+_r(f)(n_0).
  \end{align*}
Since $v_r\subset v_q$ contractively (with $r<q$), the corresponding up bound for $E_2$ will follow from the following pointwise estimate:
 \beq\label{II}
 \|\A^+(f_2)(n)-\A^+(f_2)(n_0)\|_{v_r}\lesssim M^+_r(f)(n_0), \quad \forall\; n_0\le n\le n_0+k.
 \eeq
To prove \eqref{II}, fix an increasing sequence $\{N_j\}_{j\geq 0}$ of nonnegative integers. Note that
by the definition of $f_2$, $A^+_{N_0}(f_2)(n)=0$ if $N_0\le 2k$. So we can assume that $N_0> 2k$. Then
 \begin{align*}
 \big(&A^+_{N_{j+1}}(f_2)(n) -A^+_{N_{j+1}}(f_2)(n_0)\big)-\big(A^+_{N_{j}}(f_2)(n)-A^+_{N_{j}}(f_2)(n_0)\big)\\
  &= \frac{1}{N_{j+1}+1}\sum_{i=0}^{N_{j+1}}\big(f_2(n+i)-f_2(n_0+i)\big)
  -\frac{1}{N_{j}+1}\sum_{i=0}^{N_{j}}\big(f_2(n+i)-f_2(n_0+i)\big)\\
  &= \frac{1}{N_{j+1}+1}\sum_{i=0}^{\infty}f_2(i)\big(\un_{[0,~N_{j+1}]}(i-n)- \un_{[0,~N_{j+1}]}(i-n_0)\big)\\
  &\quad -\frac{1}{N_{j}+1}\sum_{i=0}^{\infty}f_2(i)\big(\un_{[0,~N_{j}]}(i-n)- \un_{[0,~N_{j}]}(i-n_0)\big)\\
  &= \frac{1}{N_{j+1}+1}\sum_{i=0}^{\infty}f_2(i)\big(\un_{(N_j,~N_{j+1}]}(i-n)- \un_{(N_j,~N_{j+1}]}(i-n_0)\big)\\
  &\quad -\frac{N_{j+1}-N_j}{(N_{j}+1)(N_{j+1}+1)}\sum_{i=0}^{\infty}f_2(i)\big(\un_{[0,~N_{j}]}(i-n)- \un_{[0,~N_{j}]}(i-n_0)\big)\\
  &\;{\mathop=^{\rm def}}\;\alpha_j - \beta_j.
 \end{align*}
We first deal with  $\a_j$:
 \begin{align*}
 \sum_{j=0}^{\infty}\abs{\alpha_j}^r
 &=\sum_{j=0}^{\infty} \frac{1}{(N_{j+1}+1)^r}
  \big|\sum_{i=0}^{\infty}f_2(i)\big(\un_{(N_j,~N_{j+1}]}(i-n)- \un_{(N_j,~N_{j+1}]}(i-n_0)\big)\big|^r\\
 &=\sum_{j\in J_1} \frac{1}{(N_{j+1}+1)^r}
  \big|\sum_{i=0}^{\infty}f_2(i)\big(\un_{(N_j,~N_{j+1}]}(i-n)- \un_{(N_j,~N_{j+1}]}(i-n_0)\big)\big|^r\\
 &+ \sum_{j\in J_2} \frac{1}{(N_{j+1}+1)^r}
  \big|\sum_{i=0}^{\infty}f_2(i)\big(\un_{(N_j,~N_{j+1}]}(i-n)- \un_{(N_j,~N_{j+1}]}(i-n_0)\big)\big|^r\\
  &\;{\mathop=^{\rm def}}\; F_1+F_2,
 \end{align*}
where
 $$J_1=\big\{j\;:\; N_{j+1}-N_j\le n-n_0\big\}\quad\textrm{and}\quad J_2=\big\{j\;:\; N_{j+1}-N_j> n-n_0\big\}.$$
It is clear that
 $$\big|\un_{(N_j,~N_{j+1}]}(i-n)- \un_{(N_j,~N_{j+1}]}(i-n_0)\big|
 =\un_{(N_j,~N_{j+1}]}(i-n)+ \un_{(N_j,~N_{j+1}]}(i-n_0),\quad \forall\; j\in J_1.$$
Thus by the H\"older inequality
 \begin{align*}
 F_1
  &\lesssim \sum_{j\in J_1} \frac{(N_{j+1}-N_{j})^{r-1}}{(N_{j+1}+1)^r}
  \sum_{i=0}^{\infty}|f_2(i)|^r\big(\un_{(N_j,~N_{j+1}]}(i-n)+ \un_{(N_j,~N_{j+1}]}(i-n_0)\big)\\
 &\le (n-n_0)^{r-1}\sum_{j\in J_1} \frac{1}{(N_{j+1}+1)^r}
  \sum_{i=0}^{\infty}|f_2(i)|^r\big(\un_{(N_j,~N_{j+1}]}(i-n)+ \un_{(N_j,~N_{j+1}]}(i-n_0)\big)
  \end{align*}
The sum containing $\un_{(N_j,~N_{j+1}]}(i-n_0)$ is the particular case of the one containing $\un_{(N_j,~N_{j+1}]}(i-n)$ when $n=n_0$.  Hence, we need only to consider the former.  We have
  \begin{align*}
  \sum_{j\in J_1}& \frac{1}{(N_{j+1}+1)^r}
  \sum_{i=0}^{\infty}|f_2(i)|^r \un_{(N_j,~N_{j+1}]}(i-n)\\
  &\le \sum_{j=0}^{\infty} \frac{1}{(N_{j+1}+1)^r}
  \sum_{i=0}^{\infty}|f_2(i+n)|^r\un_{(N_j,~N_{j+1}]}(i) \\
  &= \sum_{i=0}^{\infty}\frac{1}{(N_{j(i)+1}+1)^r}|f_2(i+n)|^r\\
  &\le \sum_{i=2k}^{\infty}\frac{1}{(N_{j(i)+1}+1)^r}|f(i+n)|^r,
     \end{align*}
where for each $i$, $j(i)$ is the unique $j$ such that $i\in (N_j,~N_{j+1}]$.  The last sum is estimated by standard arguments:
 \begin{align*}
 \sum_{i=2k}^{\infty}\frac{1}{(N_{j(i)+1}+1)^r}\,|f(i+n)|^r
 &\le\sum_{s=1}^\infty \sum_{2^sk\leq i \leq  2^{s+1}k}\frac1{2^{rs}k^{r}}\,|f(i+n)|^r\\
 &\lesssim \sum_{s=1}^\infty \frac{1}{2^{(r-1)s}k^{r-1}}\,
           \big(\frac{1}{2^{s+2}k+1}\sum_{0\le i \leq  2^{s+2}k}\abs{f(i+n_0)}^r\big)\\
 &\le \frac{1}{k^{r-1}}\big(M^+_r(f)(n_0)\big)^r  \sum_{s=1}^\infty \frac{1}{2^{(r-1)s}}\\
 &\lesssim \frac{1}{k^{r-1}}\big(M^+_r(f)(n_0)\big)^r  .
  \end{align*}
Combining the preceding estimates and noting that $n-n_0\le k$, we get
 $$ F_1\lesssim \big(M^+_r(f)(n_0)\big)^r .$$
The second sum $F_2$ is treated in a similar way. First note that
 $$\big|\un_{(N_j,~N_{j+1}]}(i-n)- \un_{(N_j,~N_{j+1}]}(i-n_0)\big|
 =\un_{(n_0+N_j,~n+N_{j}]}(i)+ \un_{(n_0+N_{j+1},~n+N_{j+1}]}(i),\quad\forall\; j\in J_2.$$
Therefore
 \begin{align*}
 F_2
 &\lesssim (n-n_0)^{r-1}\sum_{j\in J_2} \frac{1}{(N_{j+1}+1)^r}
 \sum_{i=0}^{\infty}|f_2(i)|^r\big(\un_{(n_0+N_j,~n+N_{j}]}(i)+ \un_{(n_0+N_{j+1},~n+N_{j+1}]}(i)\big)\\
 &\lesssim (n-n_0)^{r-1}\sum_{j\in J_2} \frac{1}{(N_{j}+1)^r}
 \sum_{i=0}^{\infty}|f_2(i)|^r \un_{(n_0+N_j,~n+N_{j}]}(i)\\
 &\le (n-n_0)^{r-1}\sum_{j\in J_2}  \frac{1}{(N_{j}+1)^r}\sum_{i=0}^{\infty}|f_2(i)|^r \un_{(n_0+N_j,~n_0+N_{j+1}]}(i)\\
 &\le (n-n_0)^{r-1}\sum_{j=0}^\8  \frac{1}{(N_{j}+1)^r}\sum_{i=0}^{\infty}|f_2(i+n_0)|^r \un_{(N_j,~N_{j+1}]}(i).
 \end{align*}
Here for the next to the last inequality we have used the fact that $n+N_j<n_0+N_{j+1}$ for $j\in J_2$. Thus we again find the sum in the reasoning for $F_1$. Hence
 $$ F_2\lesssim \big(M^+_r(f)(n_0)\big)^r  .$$
Combining this estimate and the previous one for $F_1$, we get
 \beq\label{estimate-a}
 \big(\sum_{j=0}^{\infty}\abs{\alpha_j}^r\big)^{1/r} \lesssim M^+_r(f)(n_0).
 \eeq

We pass to handling $\b_j$.
 $$\beta_j=\frac{N_{j+1}-N_j}{N_{j+1}N_j}\sum_{i=0}^{\infty}f_2(i)\left(\un_{[0,~N_{j}]}(i-n)-\un_{[0,~N_{j}]}(i-n_0)\right).$$
Recall that $N_0> 2k$. So $n-n_0\le k<N_j$. Consequently,
 $$\big|\un_{[0,~N_j]}(i-n)-\un_{[0,~N_j]}(i-n_0)\big|
 =\un_{[n_0,~n)}(i)+\un_{(n_0+N_j,~n+N_j]}(i).$$
On the other hand, if $i\in [n_0,n]$,   $f_2(i)=0$. Thus we have
 \begin{align*}
 \sum_{j=0}^\8|\beta_j|^r
 &\lesssim (n-n_0)^{r-1}\sum_{j=0}^\8\frac{(N_{j+1}-N_j)^r}{N_{j+1}^rN_j^r}
 \sum_{i=0}^\8|f_2(i)|^r\un_{(n_0+N_j,~n+N_j]}(i)\\
 &\lesssim (n-n_0)^{r-1}\sum_{j=0}^\8\frac{(N_{j+1}-N_j)^r}{N_{j+1}^rN_j^r}
 (N_j+n-n_0+1)\big(\frac{1}{N_j+n-n_0+1}\sum_{i=n_0}^{n+N_j}|f(i)|^r\big)\\
 &\lesssim M^+(|f|^r)(n_0)k^{r-1}\sum_{j=0}^\8\frac{(N_{j+1}-N_j)^r}{N_{j+1}^rN_j^{r-1}}.
 \end{align*}
Here for the last inequality we have used the fact that $N_j>k\ge n-n_0$. Hence by Lemma~\ref{basic ineq}, we conclude that
 \beq\label{estimate-b}
 \big(\sum_{j=0}^\8|\beta_j|^r\big)^{1/r}\lesssim M^+_r(f)(n_0) .
 \eeq
 \eqref{estimate-a} and \eqref{estimate-b}yield
 $$\big(\sum_j\big|A_{N_{j+1}}f_2(n) -A_{N_{j+1}}f_2(n_0)\big)-\big(A_{N_{j}}f_2(n)-A_{N_{j}}f_2(n_0)\big|^r\big)^{1/r}
 \lesssim M^+_r(f)(n_0),$$
which implies \eqref{II} by taking the supremum over all increasing sequences $(N_j)$. Together with the first part of the proof, we then get \eqref{Inequality:sharp}. Thus the sufficiency of part (i)  is proved.

\medskip

(ii) Again, it suffices to prove the sufficiency. This proof  is based on Lemma~\ref{1-sided CZ}. Let  $f$ be a finitely supported function on $\ent$ and  $\lambda >0$.  Using that lemma we decompose  $f$ into its good and bad parts: $f=g+b$ with
 \begin{align*}
 &g=f \textrm{ on } \O^c\quad\textrm{and}\quad  g= \frac1{|I_i|} \sum_{j\in I_i} f(j) \, \textrm{ on } I_i  \textrm{ for each }i,\\
 &b= \sum_i b_i, \textrm{ where } b_i = \big(f -\frac1{|I_i|}\sum_{j\in I_i} f(j) \big) \un_{I_i} .
  \end{align*}
Moreover,
 \begin{enumerate}[$\bullet$]
 \item $\|g\|_\8 \le 2 \lambda$  and $\|g\|_1 \le \|f\|_1$;
\item for each $i$, $\displaystyle  \sum_{j\in\ent} b_i(j) = 0 $ and $\displaystyle \frac1{|I_i|}\sum_{j\in\ent} |b_i(j)|\le 4\lambda$.
 \end{enumerate}
We have
 $$w(\{n: \V_{q}(f) (n) > \lambda \})\le w(\{n :\mathcal{V}_{q}(g)(n) >\frac\lambda2 \})
 + w(\{n: \mathcal{V}_{q}(b) (n)> \frac\lambda2 \}).$$
The good part is easy to be estimated. Using part (i) for $p=2$  and the properties of  $g$, we obtain
 \begin{align*}
 w(\{n :\mathcal{V}_{q}(g)(n)  >\frac\lambda2\})
 &\le \frac{4}{\lambda^2}\sum_{n\in\ent} (\mathcal{V}_{q}(g))^2(n)w(n) \lesssim \frac{1}{\lambda^2} \sum_{n\in\ent} |g(n)|^2w(n)\\
  &\lesssim\frac1{\lambda} \sum_{n\in\O^c} |g(n)|w(n) +w(\O)\lesssim\frac1{\lambda} \sum_{n\in\ent} |f(n)|w(n),
 \end{align*}
where for the last inequality we have used the weak type $(1, 1)$ boundedness of $M^+$ for $A_1^+$ weights:
 \beq\label{weak}
 w(\O)\lesssim \frac1{\lambda} \sum_{n\in\ent} |f(n)|w(n).
 \eeq

We turn to treat the bad part $b$. Let $I_i=[n_i,\, n_i+k_i]$ and $\wt\O= \bigcup_i [n_i-k_i,\, n_i+k_i]$. Then
 \beq\label{b}
 w(\{n :\mathcal{V}_{q}(b)(n))> \frac\lambda2 \}) \le w (\wt\O) +
 w( \{n :n \notin \wt\O, \mathcal{V}_{q}(b) (n) > \frac\lambda2 \}).
 \eeq
The first term on the right hand side is estimated by the doubling property of $w$ and \eqref{weak}:
 \beq\label{b1}
 w (\wt\O)\lesssim \sum_i w([n_i,\, n_i+k_i]) = w(\O)
 \lesssim \frac1{\lambda} \sum_{n\in\ent} |f(n)|w(n).
 \eeq
The main part of the proof is on the second term. Let $n\notin \wt\O$. Since $b_i$ is supported on $[n_i,n_i+k_i]$ and of vanishing mean, $A^+_N(b_i)(n)=0$ if $n+N\notin [n_i,~n_i+k_i)$. Consequently, there exists at most one $i$ such that $A^+_{N}(b_{i})(n)\neq 0$. Now let $\{N_j\}_{j\ge 0}$ be  an increasing sequence. Then for every fixed $j$ there exist at most two $i$ and $i'$ such that $A^+_{N_j}(b_{i})(n)\neq 0$ and $A^+_{N_{j+1}}(b_{i'})(n)\neq 0$. Thus
 \begin{align*}
 \sum_j\big|A^+_{N_{j}}(b)(n)-A^+_{N_{j+1}}(b)(n)\big|^q &=
 \sum_j\big|A^+_{N_j}(b_{i})(n)-A^+_{N_{j+1}}(b_{i'})(n)\big|^q\\
 &\leq 2^q\sum_j\sum_i\big|A^+_{N_j}(b_{i})(n)-A^+_{N_{j+1}}(b_{i})(n)\big|^q\\
 &= 2^q\sum_i\sum_j\big|A^+_{N_j}(b_{i})(n)-A^+_{N_{j+1}}(b_{i})(n)\big|^q.
 \end{align*}
Whence
 $$\big\|\A^+(b)(n)\big\|_{v_q}^q\lesssim\sum_i\big\|\A^+(b_i)(n)\big\|_{v_1}^q.$$
Hence
 \beq\label{b2}
  w( \{n :n \notin\wt\O, \mathcal{V}_{q}(b) (n) > \frac\lambda2 \})
 \lesssim \frac{1}{\l^q}\sum_i\sum_{n\in \wt\O^c}\big\|\A^+(b_i)(n)\big\|_{v_1}^q w(n).
 \eeq
For any $n\notin\wt\O$,
 $$\big\|\A^+(b_i)(n)\big\|_{v_1}=\sum_{N\geq 0}\big|A^+_{N+1}(b_i)(n)-A^+_{N}(b_i)(n)\big|.$$
Note  that if $A^+_N (b_i)(n)\neq 0$, then $n_i-n\leq N < n_i-n+k_i$. We thus get
 \begin{align*}
 \big\|\A^+(b_i)(n)\big\|_{v_1}
 &=|A^+_{n_i-n}(b_i)(n)|+\sum_{N=n_i-n}^{n_i+k_i-n-2}\big|A^+_{N+1}(b_i)(n)-A^+_{N}(b_i)(n)\big| +\big|A^+_{n_i-n+k_i-1}(b_i)(n)\big|\\
 &=\frac{|b_i(n_i)|}{n_i-n+1}+\sum_{N=n_i-n}^{n_i+k_i-n-2}\big|A^+_{N+1}(b_i)(n)-A^+_{N}(b_i)(n)\big|+\frac{|b_i(n_i+k_i)|}{n_i-n+k_i},
 \end{align*}
However,
 \begin{align*}
 A^+_{N+1}(b_i)(n)-A^+_{N}(b_i)(n)
 &=\frac{1}{N+2}\sum_{p}b_i(j)\un_{[0,~N+1]}(j-n)-\frac{1}{N+1}\sum_{j}b_i(j)\un_{[0,~N]}(j-n)\\
     &=\frac{1}{N+2} b_i(n+N+1) - \frac{1}{(N+1)(N+2)} \sum_{j}b_i(j)\un_{[0,~N]}(j-n).
 \end{align*}
Since $N\ge n_i-n\ge k_i$ (recalling that  $n\notin \wt\O$), we then deduce
 \begin{eqnarray}\label{v1}
 \begin{array}{ccl}
  \begin{split} \begin{displaystyle}
 \big\|\A^+(b_i)(n)\big\|_{v_1}  \end{displaystyle}
 &\le \begin{displaystyle}\frac1{n_i-n}\sum_{N=n_i-n}^{n_i-n+k_i}|b_i(n+N)|\end{displaystyle}\\
 &\begin{displaystyle}\quad+\frac{1}{(n_i-n)^2}\sum_{N=n_i-n}^{n_i+k_i-n-2}\big|\sum_{j}b_i(j)\un_{[0,~N]}(j-n)\big|\end{displaystyle}\\
 &\lesssim \begin{displaystyle}\frac{1}{n_i-n}\sum_{j\in I_i} |b_i(j)| \end{displaystyle}.
 \end{split} \end{array}
 \end{eqnarray}
Hence,
 \begin{align*}
 \sum_{n\in {\wt\O^c}}\big\|\A^+(b_i)(n)\big\|_{v_1}^q w(n)
 &\lesssim\sum_{n\leq n_i-k_i}\Big[\frac{1}{n_i-n}\sum_{j\in I_i} |b_i(j)|\Big]^q w (n)\\
 &=\sum_{n\leq n_i-k_i} \frac{|I_i|^{q-1}}{(n_i-n)^q}
 \Big[\frac{1}{|I_i|}\sum_{j\in I_i} |b_i(j)|\Big]^{q-1}\Big[\sum_{j\in I_i} |b_i(j)|\Big] w (n)\\
 &\lesssim (\lambda |I_i|)^{q-1}\sum_{n\leq n_i-k_i} \frac{1}{(n_i-n)^q}w (n)
 \sum_{j\in I_i}|f(j)|.
 \end{align*}
The last double sum is estimated as follows:
 \begin{align*}
 \sum_{n\leq n_i-k_i} \frac{1}{(n_i-n)^q}w (n)\sum_{j\in I_i}|f(j)|
 &=\sum_{s=0}^\8\sum_{2^s|I_i|\le n_i-n<2^{s+1}|I_i|}\frac{1}{(n_i-n)^q}w (n)\sum_{j\in I_i}|f(j)|\\
 &\lesssim\frac1{|I_i|^{q-1}}\sum_{s=0}^\8 \frac{1}{2^{s(q-1)}} \sum_{j\in I_i}|f(j)|
 \Big[\frac{1}{|I_i| 2^{s+1}}\sum_{n_i-|I_i|2^{s+1}<n\le n_i}w (n)\Big].
 \end{align*}
Since $w$ is an $A_1^+$ weight, for any $j\in I_i=[n_i,~n_i+k_i]$ and any $s\ge0$, we have
 $$\frac{1}{|I_i| 2^{s+1}}\sum_{n_i-|I_i|2^{s+1}<n\le n_i}w (n)\lesssim w(j).$$
Therefore,
 $$\sum_{n\in {\wt\O^c}}\big\|\A^+(b_i)(n)\big\|_{v_1}^q w(n)\lesssim
 \l^{q-1}\sum_{j\in I_i}|f(j)| w(j).$$
Together with \eqref{b2}, this inequality implies
 $$w( \{n :n \notin \wt\O, \mathcal{V}_{q}(b) (n) > \frac\lambda2 \})
 \lesssim \frac1\l\sum_{j\in I_i}|f(j)|w(j).$$
Combining this with \eqref{b} and \eqref{b1}, we obtain
 $$w(\{n :\mathcal{V}_{q}(b)(n))> \frac\lambda2 \}) \lesssim\frac1\l\sum_{j\in I_i}|f(j)| w(j).$$
Along with the estimate on the good part $g$ at the beginning of the proof, this last inequality yields the desired weak type $(1, 1)$ inequality:
  $$w(\{n :\mathcal{V}_{q}(f)(n))> \lambda \}) \lesssim\frac1\l\sum_{n\in\ent}|f(n)| w(n).$$
 Thus the proof of the theorem is complete.
\cqd

\medskip

By symmetry we have the following analogue of Theorem~\ref{Thm:variation-Ap+} for $A^-_p$ weights. The definition of $A^-_p$ weights is similar to that of $A^+_p$ weights given in the introduction (see \cite{Sawyer} for more information).   Let $\A^- (f)(n)=\{A^-_N(f)(n)\}_{N\ge0}$, where
$$A^-_N(f)(n)=\frac{1}{N+1}\sum_{i=0}^{N}f(n-i).$$

\begin{cor}\label{Thm:variation-Ap-}
 Let $q>2$. Then the operator $\V_q\A^-$ maps  $\ell^p(\ent, w)$ into  $\ell^p(\ent, w)$ if and only if
   $w \in A^-_p$ with $1<p<\8$, and  $\ell^1(\ent, w)$ into $\ell^{1,\infty}(\ent, w)$ if and only if $w \in A^-_1$.
 \end{cor}

Recall that a positive function $w$ on $\ent$ belongs to $A_p$ if and only if it belongs to both $A_p^+$ and $A_p^-$ (see \cite{MT1}). Thus Theorem \ref{Thm:variation-Ap+} and Corollary \ref{Thm:variation-Ap-} imply the following

\begin{cor}\label{Thm:variation-Ap}
 Let $q>2$ and $\A (f)(n)=\{A_N(f)(n)\}_{N\ge0}$, where
$$A_N(f)(n)=\frac{1}{2N+1}\sum_{i=-N}^{N}f(n+i).$$
Then the operator $\V_q\A$ is bounded on $\ell^p(w)$ if and only if
   $w \in A_p$ with $1<p<\8$, and  from $\ell^1(\ent, w)$ into $\ell^{1,\infty}(\ent, w)$ if and only if $w \in A_1$.
\end{cor}

All results in this section equally hold for $\real$ instead of $\ent$ with essentially the same arguments.  For a locally integrable function on $\real$ and $t>0$ we define
 $$A^+_tf(x)=\frac1t\int_0^tf(x+s)ds$$
and $\A^+(f)(x)=\{A^+_tf(x)\}_{t>0}$. Here we use the same notation as in the discrete case but this should not cause any ambiguity in concrete contexts. Accordingly, define
 $$A^-_tf(x)=\frac1t\int_0^tf(x-s)ds,\quad A_tf(x)=\frac1{2t}\int_{-t}^tf(x+s)ds$$
and  $\A^-(f)(x)=\{A^-_tf(x)\}_{t>0}$,  $\A(f)(x)=\{A_tf(x)\}_{t>0}$.

\begin{thm}\label{Thm:variation-Ap cont}
 Let $q>2$. The operator $\V_q\A$  is bounded on  $L^p(\real, w)$  if and only if   $w \in A_p$ with $1<p<\8$, and from $L^1(\real, w)$ into $L^{1.\8}(\real, w)$ if and only if $w \in A_1$.\\
 A similar statement holds for the two one-sided differential operators $\A^+$ and $\A^-$.
\end{thm}

Let $\el_c^\8(\ent)$ be the subspace of $\el^\8(\ent)$ consisting of finitely supported functions. Inequality \eqref{Inequality:sharp} shows that $\V_q\A^+$ is bounded from $\el_c^\8(\ent)$ to $BMO^+(\ent)$, where $BMO^+(\ent)$ is the space of all functions $f$ on $\ent$ such that $\|f\|_{BMO^+(\ent)}=\|f^{+,\sharp}\|_\8<\8$. In fact, that inequality and its proof yield a weighted version of this $L^\8$-$BMO$ boundedness. Following \cite{MR}, given a weight $w$ define $BMO^+(\ent, w)$ to be the space of all functions $f$ such that
 $$\|f\|_{BMO^+(\ent, w)}=\sup_{n\in\ent}\, \sup_{k\in \nat}\, \sup_{n-k\le \el\le n}\,
 \frac{w(\el)}{k+1}\sum_{i=n}^{n+k}\big(f(i)- \frac1{k+1} \sum_{j=n+k}^{n +2k}f(j) \big)^+<\8.$$
The proof of \eqref{Inequality:sharp} gives the following

\begin{cor}\label{BMO:variation-Ap}
 Let $w$ be a weight on $\ent$ such that $w^{-1}\in A_1^{-}$. Then
 $$\|\V_q\A^+(f)\|_{BMO^+(\ent, w)}\lesssim \|fw\|_{\8},\quad f\in \el_c^\8(\ent).$$
\end{cor}

We have, of course, similar statements for $BMO^-$ and $BMO$. The interest of such weighted $L^\8$-$BMO$ boundedness results lies in the fact that they can be extrapolated to show the weighted type $(p, p)$ inequalities. See \cite{MR} for the ones-sided case and \cite{HMS} for the two-sided case.

\begin{rk}
 The proof of \eqref{Inequality:sharp} can be modified to show the following fact: For any $f\in \el^\8(\ent)$ either $\V_qA^+(f)\equiv\8$  or $\V_qA^+(f)<\8$ everywhere (this is easier);  see \cite{CMMTV} for the case of the Poisson semigroup on $\real^n$.  Below is an example for which the former alternative occurs:
  $$f= \sum_{k=0}^\infty \un_{(2^{2k}, 2^{2k+1}]}.$$
Indeed, consider only $\V_qA^+(f)(0)$. Then it is easy to check that there exists a constant $c>0$ such that
 $$\big|A_{2^{2k+2} } f(0) - A_{2^{2k+1} } f(0) \big|>c,\quad\forall\; k\ge0.$$
It then follows that  $\V_qA^+(f)(0)=\8$.
\end{rk}

We end this section with an application to convolution operators. For a function $\psi$ on an interval $I\subset\real$ define $\psi_t(x)=\frac1t\,\psi(\frac xt)$ and  the family of convolution operators:
 $$\psi_t*f(x)=\int_{\real}\psi_t(x-y)f(y)dy.$$
Here the integral is understood as the integral on the domain of the function $y\mapsto \psi_t(x-y)$. Let $\Psi(f)(x)=\{\psi_t*f(x)\}_{t>0}$. We will consider the $q$-variation $\V_q\Psi(f)(x)$ of $\Psi(f)(x)$. The following corollary is implicit in \cite{CJRW1} in the unweighted case.

\begin{cor}\label{Thm:variation-Ap conv}
 Let $0<a<b<\8$ and $q>2$.
 \begin{enumerate}[\rm(i)]
 \item Let $\psi$ be a differentiable function on $[a,\, b]$ such that $\int_a^bx|\psi'(x)|dx<\8$.  Then  $\V_q\Psi$  maps $L^p(\real, w)$ into $L^p(\real, w)$ for  $w \in A^-_p$ and $1<p<\8$, and  $L^1(\real, w)$ into $L^{1.\8}(\real, w)$ for $w \in A^-_1$.
  \item Let $\psi$ be a differentiable function on $[-b,\, -a]$ such that $\int_{-b}^{-a}|x|\,|\psi'(x)|dx<\8$.  Then  $\V_q\Psi$  maps $L^p(\real, w)$ into $L^p(\real, w)$ for  $w \in A^+_p$ and $1<p<\8$, and  $L^1(\real, w)$ into $L^{1.\8}(\real, w)$ for $w \in A^+_1$.
  \item Let $\psi$ be a differentiable function on $[-a,\, b]$ such that $\int_{-a}^b|x|\,|\psi'(x)|dx<\8$.  Then  $\V_q\Psi$  maps $L^p(\real, w)$ into $L^p(\real, w)$ for  $w \in A_p$ and $1<p<\8$, and  $L^1(\real, w)$ into $L^{1.\8}(\real, w)$ for $w \in A_1$.
  \end{enumerate}
  \end{cor}

 \begin{proof}
  We prove only part (i). The two others can be handled in a similar way. For $y\in [a,\, b]$ we have
 $$ \psi(y)=\psi(b)-\int_y^b\psi'(z)dz
  =\psi(b)\un_{ [a,\, b]}(y)-\int_a^b\un_{ [a,\, z]}(y)\psi'(z)dz.$$
Thus
 $$\psi_t(y)=\psi(b)\frac1t\,\un_{ [a,\, b]}(\frac yt)-\int_a^b\frac1t\,\un_{ [a,\, z]}(\frac yt)\psi'(z)dz.$$
Consequently,
 $$\psi_t*f(x)=\psi(b)\frac1t\,\un_{ [a,\, b]}(\frac \cdot t)*f(x)-\int_a^b\frac1t\,\un_{ [a,\, z]}(\frac \cdot t)*f(x)\psi'(z)dz.$$
Writing
 $$\frac1t\,\un_{ [a,\, z]}(\frac yt)=\frac1t\,\un_{ (0,\, z]}(\frac yt)-\frac1t\,\un_{ (0,\, a]}(\frac yt),$$
we get
 $$\big\|\big\{\frac1t\,\un_{ [a,\, z]}(\frac \cdot t)*f(x)\big\}_{t>0}\big\|_{v_q}
 \le (a+z)\big\|\A^-(f)(x)\big\|_{v_q}.$$
Therefore,
 $$\V_q\Psi(f)(x)
 \le \big[(a+b)|\psi(b)|+a \int_a^b|\psi'(z)|dz+ \int_a^bz|\psi'(z)|dz\big]\, \V_q\A^-(f)(x).$$
Then by Theorem~\ref{Thm:variation-Ap cont} we get the announced assertion.
 \end{proof}

 \begin{rk}
 In the preceding corollary we can take $a=0$ or $b=\8$. In the case of $b=\8$ we need impose the additional assumption that $\lim_{y\to\8}\psi(y)=0$. For instance, if $\psi$ is as in part (i) with $a=0$ and $b=\8$, then the identity
 $$ \psi(y)=-\int_y^\8\un_{ (0,\, z]}(y)\psi'(z)dz$$
allows us to get the assertion as before.
 \end{rk}


\section{Singular integrals}\label{Singular integrals}


Let $K$ be a kernel on $\real\times\real\setminus\{(x, x):x\in\real\}$. We will suppose that $K$ satisfies the following regularity conditions. There exist two constants $\d>0$ and $C>0$ such that
 \begin{itemize}
 \item[$({\rm K}_0$):] $\displaystyle |K(x, y)|\le \frac{C}{|x-y|}$ for $x\neq y$;
 \item[$({\rm K}_1$):] $\displaystyle  |K(x, y)-K(z,y)|\le \frac{C|x-z|^\d}{|x-y|^{1+\d}}$ for $|x-y|>2|x-z|$;
 \item[$({\rm K}_2$):] $\displaystyle |K(y, x)-K(y,z)|\le \frac{C|x-z|^\d}{|x-y|^{1+\d}}$ for $|x-y|>2|x-z|$.
 \end{itemize}
 By a slight abuse of notation, we will also use $K$ to denote the associated singular integral operator:
 $$K(f)(x)=\int_{\real}K(x, y)f(y)dy.$$
For any $t>0$ let $K_t$ be  the truncated operator:
 $$K_t(f)(x)=\int_{|x-y|>t}K(x, y)f(y)dy.$$
Let $\K(f)(x)=\{K_t(f)(x)\}_{t>0}$. We will consider the $q$-variation of $\K f$:
 $$\V_q\K (f)(x)=\|\K (f)(x)\|_{v_q}=\sup_{\{t_j\}}\big(\sum_{j=0}^\8\big|K_{t_j}(f)(x)-K_{t_{j+1}}(f)(x)\big|^q\big)^{1/q}.$$
For any interval $I\subset (0,\,\8)$ let
 $$R_I=\{x: x\in\real, |x|\in I\}\quad\textrm{and}\quad K_I(f)(x)=\int_{\real}K(x, y)\un_{R_I}(x-y)f(y)dy.$$
Then
  $$\V_q\K (f)(x)=\sup\big(\sum_{j=0}^\8\big|K_{I_j}(f)(x)\big|^q\big)^{1/q},$$
where the supremum runs all sequences $\{I_j\}$ of disjoint intervals of $(0,\,\8)$.

\medskip

The following is the main theorem of this section. We would emphasize that it is new even in the unweighted case. In this regard, compare it with \cite[Theorem~B]{CJRW2}.  On the other hand, the proof of part (i)  in the unweighted case provides a new proof of the weak type $(1, 1)$ variation  inequality  for the Hilbert transform of \cite{CJRW1}. It is simpler than that of \cite{CJRW1} since we do not pass through short and long variations.  Note that a similar result was proved in \cite{HLP} but only for smooth truncations.

\begin{thm}\label{vq CZ}
 Let $K$ be a  kernel satisfying $({\rm K}_0)$-$({\rm K}_2$), and let $2<q<\8$. Assume that the operator $\V_q\K$ is  of type $(p_0,p_0)$ for some $1<p_0<\8$:
  $$\int_{\real}\big(\V_q\K (f)(x)\big)^{p_0}dx\lesssim \int_{\real}|f(x)|^{p_0}dx,\quad\forall\; f\in L^{p_0}(\real).$$
Then
 \begin{enumerate}[\rm(i)]
 \item for $w\in A_1$
  $$w\big(\big\{x:\V_q\K (f)(x)>\l\big\}\big)\lesssim\frac1\l \int_{\real}|f(x)|w(x)dx,\quad\forall\; f\in L^1(\real, w),\;\forall\; \l>0;$$
 \item for $1< p<\8$ and $w\in A_p$
  $$\int_{\real}\big(\V_q\K (f)(x)\big)^pw(x)dx\lesssim \int_{\real}|f(x)|^pw(x)dx,\quad\forall\; f\in L^p(\real, w).$$
   \end{enumerate}
\end{thm}

\begin{proof}

This proof is similar to that of Theorem~\ref{Thm:variation-Ap+}. Technically,  it is slightly simpler for the kernel $K$ satisfies the regularity conditions $({\rm K}_0)$-$({\rm K}_2$) while the kernel of the differential operators do not. For clarity, we divide the proof into three steps.

\medskip\n{\bf Step~1.} The main objective of this step is to show that  $\V_q\K$ is  of type $(p,p)$ for every $1<p<p_0$. To this end we will prove the weak type $(1, 1)$ of  $\V_q\K$, that is the unweighted version of part (i). The full generality of (i) will be treated in step~3. However, to avoid too many repetitions, we will present most of our arguments for a general weight $w\in A_1$ in this first step. Only in one place we  need to assume that $w\equiv1$.

The main tool is again the  classical Calder\'on-Zygmund decomposition.
Let  $f$ be a compactly supported integrable function on $\real$ and  $\lambda >0$.  Then  $\O = \{x\in\real: M(f)(x) > \lambda\}$ can be decomposed into (finitely many) disjoint intervals: $\O = \bigcup_iI_i$ with the following properties
 \begin{enumerate}[$\bullet$]
 \item $|f| \le \lambda$ on $\O^c$;
 \item $\displaystyle |\O| \le \frac1{\lambda}\,\|f\|_1$;
 \item $\displaystyle \lambda < \frac1{|I_i|} \int_{I_i} |f| \le 2 \lambda$.
 \end{enumerate}
Accordingly,   $f=g+b$ with
 \begin{align*}
 &g=f \textrm{ on } \O^c\quad\textrm{and}\quad  g= \frac1{|I_i|} \int_{I_i} f \, \textrm{ on } I_i  \textrm{ for each }i,\\
 &b= \sum_i b_i, \textrm{ where } b_i = \big(f -\frac1{|I_i|} \int_{I_i} f  \big) \un_{I_i} .
  \end{align*}
It is clear that
 \begin{enumerate}[$\bullet$]
 \item $\|g\|_\8 \le 2 \lambda$;
\item for each $i$, $\displaystyle  \int_{\real} b_i= 0 $ and $\displaystyle \frac1{|I_i|}\int_{\real}|b_i|\leq 4\lambda$.
 \end{enumerate}

To estimate $w(\{x: \V_{q}\K(f) (x) > \lambda \})$, by rescaling, we can assume that $\l=1$. With the above notation (with $\l=1$),  we have
 $$
 w(\{x: \V_{q}\K(f) (x) > 1 \})\le w(\{x :\V_{q}\K(g)(x) >\frac12 \})
 + w(\{x: \V_{q}\K(b) (x)> \frac12 \}).
 $$
We have to estimate the two terms on the right. It is only here for the good part that we require that $w\equiv1$. Thus if $w\equiv1$, then by the $L^{p_0}$-boundedness of $\V_q\K$,  we have
 \begin{align*}
 w(\{x :\V_{q}\K(g)(x) >\frac12 \})
 &=|\{x :\V_{q}\K(g)(x) >\frac12 \}|
 \le 2^{p_0}\int_{\real}(\V_{q}\K(g)(x) )^{p_0}dx\\
 &\lesssim \int_{\real} |g(x)|^{p_0}dx\lesssim \int_{\real} |f(x)|dx.
 \end{align*}

In the rest of this step, we assume that $w$ is a general weight in $A_1$. To treat the bad part $b$, let $\wt\O= \bigcup_i \wt I_i$, where  $\wt I$  denotes the interval with the same center as $I$ but three times the length. Then
 \begin{align*}
 w(\{x: \V_{q}\K(b) (x)> \frac12 \})
 &\le w(\wt\O) + w( \{x \in \wt\O^c:\V_{q}\K(b) (x) > \frac12 \})\\
 &\lesssim \int_{\real} |f(x)|w(x)dx + w( \{x \in \wt\O^c: \V_{q}\K(b) (x) > \frac12 \}).
 \end{align*}
Here we have used the doubling property of $w$ and the weak type $(1,1)$ of $M$ relative to $A_1$ weights for $w(\wt\O)$:
 $$w(\wt\O)\lesssim w(\O)\lesssim \int_{\real} |f(x)|w(x)dx.$$
 It thus remains to estimate the last term. This is  our main task in this part of the proof.

For every $x\notin \wt\O$ choose an increasing sequence $\{t_j\}$ (that depends  on $x$) such that
 $$\V_{q}\K(b) (x)\le 2\big(\sum_j\big|K_{(t_j,\,t_{j+1}]}(b)(x)\big|^q\big)^{1/q}.$$
Let $R_j=R_{(t_j,\,t_{j+1}]}$ and $K_{R_j}=K_{(t_j,\,t_{j+1}]}$. Note that $K_{R_j}(b)(x)\neq0$ only if $x+R_j$ meets some $I_i$. We consider two cases:
 $$\I^1_j=\{i: I_i\subset  x+R_j\}\quad\textrm{and}\quad
 \quad \I^2_j=\{i:  I_i\not\subset x+R_j, I_i\cap (x+R_j)\neq\emptyset\}.$$
Then
 $$\V_{q}\K(b) (x)\le
  2\big(\sum_{j}\big|\sum_{i\in\I^1_j}K_{R_j}(b_i)(x)\big|^q\big)^{1/q}
   + 2\big(\sum_{j}\big|\sum_{i\in\I^2_j}K_{R_j}(b_i)(x)\big|^q\big)^{1/q}.$$
It follows that
 \begin{eqnarray}\label{b1b sym}
 \begin{array}{ccl}
  \begin{split} \begin{displaystyle}
  w(\{x \in \wt\O^c: \V_{q}\K(b) (x) > \frac12 \})
  \end{displaystyle}
  &\le\begin{displaystyle}
  w(\{x \in \wt\O^c: \sum_{j}\big|\sum_{i\in\I^1_j}K_{R_j}(b_i)(x)\big|^q> \frac1{8^q} \})\end{displaystyle}\\
   &\begin{displaystyle}+w(\{x \in \wt\O^c: \sum_{j}\big|\sum_{i\in\I^2_j}K_{R_j}(b_i)(x)\big|^q> \frac1{8^q} \})
   \end{displaystyle}.
 \end{split} \end{array}
 \end{eqnarray}
  Let  $i\in\I^1_j$, that is, $I_i\subset  x+R_j$. Since $b_i$  is of vanishing mean, we have
 $$K_{R_j}(b_i)(x)=\int_{\real}K(x, y)b_i(y)dy=\int_{\real}\big(K(x, y)-K(x, c_i)\big)b_i(y)dy,$$
where $c_i$ is the center of $I_i$. Therefore, by $({\rm K}_2$)
 \begin{align*}
 \big(\sum_{j}\big|\sum_{i\in\I^1_j}K_{R_j}(b_i)(x)\big|^q\big)^{1/q}
 &\le\sum_{j}\big|\sum_{i\in\I^1_j}K_{R_j}(b_i)(x)\big|\\
 &\le\sum_i\int_{\real}\big|K(x, y)-K(x, c_i)\big|\,|b_i(y)|dy\\
 &\lesssim \sum_i\frac{|I_i|^\d}{|x-c_i|^{1+\d}}\int_{\real}|b_i(y)|dy\\
 &\lesssim \sum_i\frac{|I_i|^\d}{|x-c_i|^{1+\d}}\int_{I_i}|f(y)|dy.
 \end{align*}
Thus we deduce
 \begin{align*}
  w(\{x \in \wt\O^c: \sum_{j}\big|\sum_{i\in\I^1_j}K_{R_j}(b_i)(x)\big|^q > \frac1{8^q} \})
  &\lesssim\int_{\wt\O^c}\sum_{j\in\I^1_j}\big|K_{R_j}(b_i)(x)\big|w(x)dx\\
 &\lesssim \int_{\wt\O^c}\sum_i\frac{|I_i|^\d}{|x-c_i|^{1+\d}}\int_{I_i}|f(y)|dy\,w(x)dx\\
 &\lesssim\sum_i|I_i|^d\int_{I_i}|f(y)|\big[\int_{|x-c_i|>|I_i|}\frac{w(x)dx}{|x-c_i|^{1+\d}}\big]dy.
  \end{align*}
 However, by the definition of $A_1$ weights, we have
  \begin{align*}
  \int_{|x-c_i|>|I_i|}\frac{w(x)dx}{|x-c_i|^{1+\d}}
  &\lesssim |I_i|^{-\d}\sum_{s=0}^\8 2^{-\d s}\big[\frac1{2^{s}|I_i|}\int_{2^s|I_i|<|x-c_i|\le 2^{s+1}|I_i|}|w(x)dx\big]\\
  &\lesssim |I_i|^{-\d}\sum_{s=0}^\82^{-\d s}\big[\frac1{2^{s+2}|I_i|}\int_{|x-c_i|\le 2^{s+1}|I_i|}w(x)dx\big]\\
  & \lesssim |I_i|^{-\d} w(y)\;\textrm{ for }\; a.e. \; y\in I_i.
   \end{align*}
 Thus
  \beq\label{b1bb sym}
  w(\{x \in \wt\O^c: \sum_{j}\big|\sum_{i\in\I^1_j}K_{R_j}(b_i)(x)\big|^q > \frac1{8^q} \})\
   \lesssim \int_{\real}|f(y)|w(y)dy.
   \eeq

For the part on $\I^2_j$ is more delicate. A simple geometrical inspection shows that $\I^2_j$ contains at most four points for any $j$.  It then follows that
  $$
 \sum_{j}\big|\sum_{i\in\I^2_j}K_{R_j}(b_i)(x)\big|^q
 \le 4^{q-1} \sum_{j}\sum_{i\in\I^2_j}|K_{R_j}(b_i)(x)|^q
 \le 4^{q-1}\sum_i \big(\sum_{j}|K_{R_j}(b_i)(x)|\big)^q.
$$
 Thus
 $$w(\{x \in \wt\O^c: \sum_{j}\big|\sum_{i\in\I^2_j}K_{R_j}(b_i)(x)\big|^q> \frac1{8^q} \})
 \lesssim \sum_i\int_{\wt\O^c}\big(\sum_{j}|K_{R_j}(b_i)(x)|\big)^qw(x)dx.$$
Let $x\notin \wt\O$.  Then by $({\rm K}_0)$
 \begin{align*}
 \sum_{j}|K_{R_j}(b_i)(x)|
 &\le \sum_{j}\int_{R_j}\big|K(x, y)\big|\,|b_i(y)|dy\\
 &\lesssim \frac1{|x-c_i|}\sum_{j}\int_{\real}\un_{x+R_j}(y)|b_i(y)|dy\\
 &\lesssim \frac1{|x-c_i|}\int_{I_i}|b_i(y)|dy\\
 &\lesssim \frac1{|x-c_i|}\int_{I_i}|f(y)|dy.
 \end{align*}
Therefore, as in the display of inequalities following \eqref{v1}, we have
  \begin{align*}
 \int_{\wt\O^c}\big(\sum_{j}|K_{R_j}(b_i)(x)|\big)^qw(x)dx
 &\lesssim  \int_{\wt\O^c}\frac{|I_i|^{q-1}}{|x-c_i|^q}\int_{I_i}|f(y)|dyw(x)dx\\
 &\lesssim |I_i|^{q-1}\int_{I_i}|f(y)|\big[ \int_{|x-c_i|>|I_i|}\frac{1}{|x-c_i|^q}w(x)dx\big]dy\\
 &\lesssim \int_{I_i}|f(y)|w(y)dy.
 \end{align*}
Hence
 $$w(\{x \in \wt\O^c: \sum_{j}\big|\sum_{i\in\I^2_j}K_{R_j}(b_i)(x)\big|^q > \frac1{8^q} \})
 \lesssim \int_{\real}|f(y)|w(y)dy.$$
Combining this with \eqref{b1b sym} and \eqref{b1bb sym}, we get
 $$ w(\{x \in \wt\O^c: \V_{q}\K(b) (x) > \frac12 \}) \lesssim \int_{\real}|f(y)|w(y)dy.$$
This implies the desired estimate on the bad part $b$:
 $$w(\{x: \V_{q}\K(b) (x)> \frac12 \})\lesssim \int_{\real}|f(y)|w(y)dy.$$
Combining this with the unweighted estimate for the good part in the beginning of this step, we get the unweighted weak type $(1, 1)$ of $\V_q\K$.

Therefore, by the Marcinkiewicz interpolation theorem, we deduce that $\V_q\K$ is of type $(p, p)$   for any $1<p<p_0$.

\medskip\n{\bf Step~2.} This step is devoted to the proof of part (ii). As in the proof of Theorem~\ref{Thm:variation-Ap+}, it suffices to show the following inequality for any compactly supported function $f$
 \begin{equation}\label{Inequality:sharpSI}
  (\V_q\K(f))^{\sharp}\lesssim M_r(f)
 \end{equation}
for $r>1$ close to $1$ ($r<\min(p_0,\, q)$). Note that the symmetric analogue of Lemma~\ref{Le:malaga} can be found in many  books on real variable harmonic analysis, for instance, \cite[Theorem~IV.2.20]{gar-rubio}. Fix a function $f$ on $\real$ and a point $x_0\in\real$. Recall that
 $$( \V_q\K(f))^{\sharp}(x_0)=
 \sup_{I} \frac{1}{|I|} \int_I\Big| \V_q\K(f)(x)-\frac{1}{|I|}\int_I  \V_q\K(f)(y)dy\Big|dx,$$
where the supremum runs over all intervals $I$ containing $x_0$. Fix such an interval $I$ and write $f=f_1+f_2$ with $f_1=f\un_{\wt I}$ and $f_2=f\un_{\wt I^c}$. Then
 \begin{align*}
  \frac{1}{|I|}& \int_I\Big| \V_q\K(f)(x)-\frac{1}{|I|}\int_I  \V_q\K(f)(y)dy\Big|dx\\
 &\le  \frac{2}{|I|} \int_I\Big|\ \V_q\K(f)(x)- \V_q\K(f_2)(c)\Big|dx\\
 &\le \frac{2}{|I|} \int_I \V_q\K(f_1)(x)dx
 + \frac{2}{|I|} \int_I\|\K(f_2)(x)-\K(f_2)(c)\|_{v_q}dx\\
 &\;{\mathop=^{\rm def}}\;  D_1+D_2 ,
 \end{align*}
where $c$ is the center of $I$. By the H\"older inequality and the $L^r$-boundedness of $\mathcal V_q$ already already proved in step~1, we get
  $$D_1\lesssim  \Big(\frac{1}{|I|}\int_I\big( \V_q\K(f_1)(x)\big)^rdx\Big)^{1/r}
   \lesssim \Big(\frac{1}{|I|}\int_I|f_1(x)|^r\Big)^{1/r}
   \lesssim M_r(f)(x_0).$$
To prove the corresponding up bound for $D_2$ it suffices to show
 \beq\label{D2}
 \big\|\K(f_2)(x)-\K(f_2)(c)\big\|_{v_r}\lesssim M_r(f)(x_0), \quad \forall\; x\in I.
 \eeq
To this end, fix an increasing sequence $\{t_j\}_{j\geq 0}$  and keep the meaning of $R_j$ and $K_{R_j}$ as in the proof of (i). Then
 \begin{align*}
  K_{R_j}(f_2)(x)-K_{R_j}(f_2)(c)
  &=\int_{\real}\big(K(x, y)\un_{R_j}(x- y)-K(c, y)\un_{R_j}(c-y)\big)f_2(y)dy\\
  &=\int_{\real}\big(K(x, y)-K(c, y)\big)\un_{R_j}(x-y)f_2(y)dy\\
  &\quad+\int_{\real}K(c, y)\big(\un_{R_j}(x-y)-\un_{R_j}(c-y)\big)f_2(y)dy\\
  &\;{\mathop=^{\rm def}}\; a_j+b_j.
  \end{align*}
The first term $a_j$ is easy to be handled. Indeed, by $({\rm K}_1)$
 \begin{align*}
 \big(\sum_{j=0}^{\infty}|a_j|^r\big)^{1/r}
 &\le \sum_{j=0}^{\infty}|a_j|\\
 &\le  \sum_{j=0}^{\infty}\int_{\real}\big|K(x, y)-K(c, y)\big|\un_{R_j}(x-y)|f_2(y)|dy\\
 &\lesssim \sum_{j=0}^{\infty}\int_{\real}\frac{|x-c|^\d}{|y-c|^{1+\d}}\,\un_{R_j}(x-y)|f_2(y)|dy\\
 &\le |x-c|^\d\int_{|y-c|>|I|}\frac{1}{|y-c|^{1+\d}}|f(y)|dy\\
 &\lesssim M(f)(x_0)\le  M_r(f)(x_0).
 \end{align*}
To deal with the second term $b_j$ we introduce, as in the proof of Theorem~\ref{Thm:variation-Ap+}, the following sets
 $$J_1=\big\{j\;:\; t_{j+1}-t_j\le |x-c|\big\}\quad\textrm{and}\quad J_2=\big\{j\;:\; t_{j+1}-t_j> |x-c|\big\}.$$
Then
 \begin{align*}
 \big|\un_{R_j}(x-y)-\un_{R_j}(c-y)\big|
 &\le \un_{R_j}(x-y)+\un_{R_j}(c-y), \;j\in J_1;\\
 \big|\un_{R_j}(x-y)-\un_{R_j}(c-y)\big|
 &\le \un_{R_{(t_j,\, t_j+|x-c|}]}(x-y) +\un_{R_{(t_{j+1}, \,t_{j+1}+|x-c|]}}(x-y)\\
 &+\un_{R_{(t_j,\, t_j+|x-c|]}}(c-y) +\un_{R_{(t_{j+1}, \,t_{j+1}+|x-c|]}}(c-y), \;j\in J_2.
  \end{align*}
We first consider the part on $J_1$. By $({\rm K}_0)$ and the H\"older inequality
 \begin{align*}
 \sum_{j\in J_1}|b_j|^r
 &\lesssim \sum_{j\in J_1}\big( \int_{\real}\big|K(c, y)|\big(\un_{R_j}(x-y)+\un_{R_j}(c-y)\big)|f_2(y)|dy\big)^r\\
 &\lesssim \sum_{j\in J_1} (t_{j+1}-t_j)^{r-1}\int_{\real}\frac1{|y-c|^r}\,\big(\un_{R_j}(x-y)+\un_{R_j}(c-y)\big)|f_2(y)|^rdy\\
 &\lesssim |x-c|^{r-1}\int_{|y-c|>|I|}\frac1{|y-c|^r}|f(y)|^rdy\lesssim (M_r(f)(x_0))^r.
  \end{align*}
The part on $J_2$ is treated in a similar way:
 \begin{align*}
 \sum_{j\in J_2}|b_j|^r
 &\lesssim \sum_{j\in J_2} \big(\int_{\real}\big|K(c,y)|\big(\un_{R_{(t_j, t_j+|x-c|]}}(x-y) +\un_{R_{(t_j, t_j+|x-c|]}}(c-y) \big)|f_2(y)|dy\big)^r\\
 &\lesssim |x-c|^{r-1}\int_{|y-c|>|I|}\frac1{|y-c|^r}|f(y)|^rdy\lesssim (M_r(f)(x_0))^r,
  \end{align*}
where we used the fact that $\{x+R_{(t_j, t_j+|x-c|]}\}_{j\in J_2}$ is a disjoint family of subsets of $\real$. Combining the preceding inequalities, we get \eqref{D2}.

\medskip\n{\bf Step~3.} We go back to the full generality of part (i). This part is almost proved in step~1. The only missing point is the weighted estimate for the good part $g$ there. The ingredient for this estimate is the weighted type $(p_0, p_0)$ of $\V_q\K$ with respect to any weighted $w\in A_{p_0}$. Now step~2 makes this at our disposal. Thus the missing point is fixed up in the weighted case, so we have proved part (i). Thus the proof of the theorem is complete.
 \end{proof}

\begin{rk}
The two-sided version of Corollary~\ref{BMO:variation-Ap} admits an analogue for singular integrals thanks to \eqref{Inequality:sharpSI}.
\end{rk}



In the rest of this section we give two  important examples  to which Theorem~\ref{vq CZ} applies.

\medskip\n{\bf Hilbert transform.} The first variation inequalities for singular integral operators are those for the Hilbert transform $H$ proved in \cite{CJRW1}. Let $H$ and $H_t$ be as in \eqref{H} and  \eqref{Ht}.
Let $\H (f)(x)=\{H_t(f)(x)\}_{t>0}$. The main theorem of \cite{CJRW1} asserts that for any $2<q<\8$ the operator $\V_q\H$ is  of type $(p,p)$ for $1<p<\8$ and weak type $(1, 1)$. Then applying Theorem~\ref{vq CZ} to $H$, we get the following

\begin{cor}\label{vq H}
 Let $2<q<\8$. Then for $1<p<\8$ and $w\in A_p$
 $$\int_{\real}\big(\V_q\H(f)(x)\big)^pw(x)dx\lesssim \int_{\real}|f(x)|^pw(x)dx$$
 and for $w\in A_1$
  $$w\big(\big\{x:\V_q\H(f)(x)>\l\big\}\big)\lesssim\frac1\l \int_{\real}|f(x)|w(x)dx.$$
 \end{cor}

What is new in this statement is the weak type $(1, 1)$ since the type $(p, p)$ was already proved in \cite{GT}. In view of the importance of the Hilbert transform, we wish to present an alternative simple proof of the above corollary by using  vector-valued kernel techniques. The main interest of this new proof is on the weak type $(1,1)$ inequality. Incidentally, this new approach gives a much simpler proof of the unweighted weak type $(1, 1)$ inequality  of  \cite{CJRW1}.

Let $\f$ be a $C^2$ function on $\real$ such that $\un_{[3/2,\,\8)}\le \f\le\un_{[1/2,\,\8)}$. Consider the smooth truncation of $H$:
 $$\wt H_t(f)(x)=\int_{\real}\f(\frac{|x-y|}t)\frac{f(y)}{x-y}\,dy.$$
Let $\wt\H (f)(x)=\{\wt H_t(f)(x)\}_{t>0}$. We will view $\wt\H$ as a Calder\'on-Zygmund singular integral with a vector-valued kernel. To this end define $H_\f(x)=\{\f(\frac{|x|}t)\frac{1}{x}\}_{t>0}$. Then $H_\f$ is a kernel on $\real$ with values in $v_q$, as shown by the following simple lemma.

\begin{lem}
 $H_\f$ is a regular Calder\'on-Zygmund  $v_q$-valued kernel in the following sense
 $$\|H_\f(x)\|_{v_q}\lesssim \frac1{|x|}\quad\textrm{and}\quad \|H_\f(x+y)-H_\f(x)\|_{v_q}\lesssim \frac{|y|}{|x|^2},\quad\forall\; |x|>2|y|.$$
\end{lem}

\begin{proof}
 Note that for a differentiable function $\phi$ on $(0,\,\8)$
 $$\|\phi\|_{v_q}\le\|\phi\|_{v_1}=\int_0^\8|\phi'(t)|dt.$$
Then
 $$\|H_\f(x)\|_{v_q}\le  \frac1{|x|} \int_0^\8|\f'(t)|dt$$
and
 \begin{align*}
 \|H_\f(x+y)-H_\f(x)\|_{v_q}
 &\le  \int_0^\8\big|\f'(t|x+y|)-\f'(t|x|)\big|dt\\
 &\le|y|\int_0^\8\int_0^1|\f''(t(1-s)|x|+ts|x+y|)|tds\,dt\\
 &\le |y|\int_0^1\int_0^\8\frac{|\f''(t)|t}{((1-s)|x|+s|x+y|)^2}dt\,ds\\
 &\lesssim \frac{|y|}{|x|^2}\int_0^\8 |\f''(t)|tdt.
 \end{align*}
\end{proof}

\n\emph{An alternative proof of Corollary~\ref{vq H}.} For $t>0$ we have
  $$(H_t- \wt H_t)(f)(x)=\int_{\real}\frac1y (\un_{[1,\,\8)}-\f)(|y|)f(x-ty)dy
 =\psi_t^+*f(x)+\psi_t^-*f(x),$$
 where
 $$\psi^+(y)=\frac1y (\un_{[\frac12,\,\frac32]}-\f\un_{[\frac12,\,\frac32]})(y)\textrm{ and }
 \psi^-(y)=\frac1y (\f\un_{[-\frac32,\,-\frac12]}-\un_{[-\frac32,\,-\frac12]})(y).$$
Thus Corollary~\ref{Thm:variation-Ap conv} insures that the operators $\V_q\H$ and $\V_q\wt\H$ have the same boundedness properties on $L^p(\real, w)$.  Hence we need only to show Corollary~~\ref{vq H} for $\V_q\wt\H$ instead of $\V_q\H$.

Now using the kernel $H_\f$ we can write
 $$\wt\H(f)(x)=H_\f*f(x)=\int_{\real}H_\f(x-y)f(y)dy.$$
Namely, $\wt\H$ is a Calder\'on-Zygmund operator with the $v_q$-valued kernel $H_\f$. By standard Calder\'on-Zygmund  theory for vector-valued kernels (cf. eg. \cite{RTR} and section~IV.3 of \cite{gar-rubio}),  we see that if $\wt\H$ is bounded from $L^{p_0}(\real)$ to $L^{p_0}(\real;v_q)$ for some $1<p_0<\8$, then $\wt\H$ is bounded from $L^{p}(\real, w)$ to $L^{p}(\real, w;v_q)$ for any $1<p<\8$ and $w\in A_p$, and from $L^{1}(\real, w)$ to $L^{1,\8}(\real, w;v_q)$ for any $w\in A_1$.  However, the type $(2, 2)$ of  $\V_q\H$ proved in \cite{CJRW1} implies the same property of $\V_q\wt\H$. Therefore, to conclude the proof, it remains to note that $\|\wt\H(f)(x)\|_{v_q}=\V_q\wt\H(f)(x)$. \cqd

\medskip\n{\bf Cauchy integrals on Lipschitz curves.} Let $\f:\real\to\real$ be a Lipschitz function. The Cauchy integral on the graph of $\f$ is defined by
 $$\int_{\real}\frac{1+i\f'(y)}{x-y+i(\f(x)-\f(y))}\,f(y)dy.$$
Since the multiplication by the bounded function $1+\f'$ does not affect any boundedness property of the Cauchy integral, we introduce the following regular Calder\'on-Zygmund kernel
 $$C^\f(x)=\frac1{x+i\f(x)}, $$
the associated  singular integral and truncations
 $$C^\f(f)(x)=\int_{\real}C^\f(x-y)f(y)dy,\quad C^\f_t(f)(x)=\int_{|x-y|>t}C^\f(x-y)f(y)dy.$$
The celebrated paper \cite{CMM} shows that $C^\f$ is bounded on $L^2(\real)$ (so on $L^p(\real)$ too for any $1<p<\8$). As before, let $\C^\f(f)(x)=\{C^\f_t(f)(x)\}_{t>0}$ and consider the $q$-variation operator $\V_q\C^\f$. It is proved in \cite{MTo2} that this operator is bounded on $L^2(\real)$ (see also \cite{MTo} and \cite{Mas} for related results). Thus Theorem~\ref{vq CZ} implies the following

\begin{cor}\label{vq C}
 Let $2<q<\8$. Then the operator $\V_q\C^\f$ is bounded on $L^p(\real, w)$ for $1<p<\8$ and $w\in A_p$, and from $L^1(\real, w)$  to $L^{1, \8}(\real, w)$ for  $w\in A_1$.
 \end{cor}


\section{Vector-valued extension}\label{sect-vector}


In this section  we will show that all results in the previous two admit analogues for functions taking values in $\el^\rho$ for $1<\rho<\8$. We will do this only for Theorems~\ref{Thm:variation-Ap+} and \ref{vq CZ}. Note that these vector-valued results are new in the unweighted case too. Even in this latter case, their proofs depend on the previous weighted variation inequalities in the scalar case.

\medskip

The following result improves Fefferman-Stein's celebrated vector-valued maximal inequality \cite{FS} in the one-dimensional case. Note that the one-sided analogue of Fefferman-Stein's inequality is proved in \cite{GS} for $1<p<\8$ and in \cite{X} for $p=1$.

\begin{thm}\label{Thm:variation-Ap+vector}
 Let $q>2$ and $1<\rho<\8$.
\begin{enumerate}[\rm (i)]
  \item Let $1<p<\infty$ and $w \in A^+_p$. Then for any finite sequence $\{f_k\}_{k\ge1}\subset\ell^p(\ent, w)$
  $$\sum_{n\in\ent}\Big(\sum_{k\ge1}\big(\V_q \A^+(f_k)(n)\big)^\rho\Big)^{p/\rho}w(n)
  \lesssim \sum_{n\in\ent}\Big(\sum_{k\ge1}|f_k(n)|^\rho\Big)^{p/\rho}w(n).$$
  \item  Let $w \in A^+_1$. Then for any finite sequence $\{f_k\}_{k\ge1}\subset\ell^1(\ent, w)$ and $\l>0$
  $$w\big(\big\{n\in\ent: \sum_{k\ge1}\big(\V_q \A^+(f_k)(n)\big)^\rho>\l^\rho\big\}\big)
  \lesssim \frac1\l\sum_{n\in\ent}\Big(\sum_{k\ge1}|f_k(n)|^\rho\Big)^{1/\rho}w(n).$$
\end{enumerate}
 \end{thm}

\begin{proof}
 The case $p=\rho$ reduces to Theorem~\ref{Thm:variation-Ap+}. Thus by the extrapolation theorem of \cite{MR}, we get part (i) for any $1<p<\8$. Note however that the vector-valued case is not covered by  \cite{MR}. But one can easily check that the extrapolation theorem there extends to the vector-valued case too. Regarding this vector-valued extension of extrapolation techniques, see also section~V.6 of \cite{gar-rubio}.

\medskip

It remains to prove the weak type $(1, 1)$ inequality. This  proof is similar to that of part (ii) of Theorem~\ref{Thm:variation-Ap+}, so we will be very brief. Let $f=\{f_k\}_k$.  $f$ is viewed as a function from $\ent$ to $\el^\rho$. We now apply Lemma~\ref{1-sided CZ} to $\l$ and the function $\f$ defined by $\f(n)=\|f(n)\|_{\el^\rho}$. Keeping the notation in that lemma and in the proof of Theorem~\ref{Thm:variation-Ap+}, (ii), we then have $f=g+b$. Both $g$ and $b$ are, of course, $\el^\rho$-valued functions on $\ent$. As in the proof of Theorem~\ref{Thm:variation-Ap+}, the good part $g$ is estimated by (ii) in the case $p=\rho$. The bad part $b$ inside the subset $\wt\O$ is treated in the same way as before. Thus our remaining task is to show (with  $b=\{b_k\}_k$)
 $$
 w\big(\big\{n\in\wt\O^c: \sum_k\big(\V_q \A^+(b_k)(n)\big)^\rho>\frac{\l^\rho}{2^\rho}\big\}\big)
  \lesssim \frac1\l\sum_{n\in\ent}\Big(\sum_k|f_k(n)|^\rho\Big)^{1/\rho}w(n).
 $$
Recall that $b=\sum_ib_i$ and each $b_i$ is supported on the interval $I_i$ and of vanishing mean.  So the same is true for $b_k$:
 $$b_k=\sum_ib_{i,k},\quad {\rm supp}(b_{i,k})\subset I_i, \quad\sum_{n\in\ent}b_{i,k}(n)=0.$$
Thus for any $n\in\wt\O^c$ we have again
 $$\big\|\A^+(b_k)(n)\big\|_{v_q}\lesssim\big(\sum_i\big\|\A^+(b_{i,k})(n)\big\|_{v_1}^q\big)^{1/q}.$$
Now choose $r$ such that $1<r\le\min(q,\;\rho)$. Then we get
  \begin{align*}
  \big(\sum_k\big(\V_q \A^+(b_k)(n)\big)^\rho\big)^{1/\rho}
  &\lesssim  \big(\sum_k(\sum_i\big\|\A^+(b_{i,k})(n)\big\|_{v_1}^q\big)^{\rho/q}\big)^{1/\rho}\\
  &\lesssim  \big(\sum_i(\sum_k\big\|\A^+(b_{i,k})(n)\big\|_{v_1}^\rho\big)^{r/\rho}\big)^{1/r}.
  \end{align*}
Hence
 $$
  w\big(\big\{n\in\wt\O^c: \sum_k\big(\V_q \A^+(b_k)(n)\big)^\rho>\frac{\l^\rho}{2^\rho}\big\}\big)
 \lesssim \frac{1}{\l^r}\sum_i\sum_{n\in \wt\O^c}(\sum_k\big\|\A^+(b_{i,k})(n)\big\|_{v_1}^\rho\big)^{r/\rho} w(n).
$$
However, for any $i$ and $n\in\wt\O^c$ with  $n\le n_i-k_i$ by \eqref{v1},   we have
 $$\|\A^+(b_{i,k})(n)\|_{v_1} \lesssim \frac{1}{n_i-n}\sum_{j\in I_i} |b_{i,k}(j)|.$$
So
  $$(\sum_k\|\A^+(b_{i,k})(n)\|_{v_1} ^\rho\big)^{1/\rho}\lesssim \frac{1}{n_i-n}\sum_{j\in I_i}\big (\sum_k|b_{i,k}(j)|^\rho\big)^{1/\rho}
  = \frac{1}{n_i-n}\sum_{j\in I_i}\|b_i\|_{\el^\rho}.$$
The rest of the proof is exactly the same as the corresponding part of the proof of Theorem~\ref{Thm:variation-Ap+} just with the norm $\|\,\|_{\el^\rho}$ instead of the absolute value. We omit the details.
\end{proof}

\begin{thm}\label{vq CZ vector}
 Let $2<q<\8$ and $1<\rho<\8$. Let $K$ be a  kernel satisfying the assumption of Theorem~\ref{vq CZ}.  Then
 \begin{enumerate}[\rm(i)]
 \item for $w\in A_1$
  $$w\big(\big\{x:\big(\sum_k\big(\V_q\K (f_k)(x)\big)^\rho\big)^{1/\rho}>\l\big\}\big)
  \lesssim\frac1\l \int_{\real}\big(\sum_k|f_k(x)|^\rho\big)^{1/\rho}w(x)dx,\; f_k\in L^1(\real, w), \l>0;$$
 \item for $1< p<\8$ and $w\in A_p$
  $$\int_{\real}\big(\sum_k\big(\V_q\K (f_k)(x)\big)^\rho\big)^{p/\rho}w(x)dx
  \lesssim \int_{\real}\big(\sum_k|f_k(x)|^\rho\big)^{p/\rho}w(x)dx,\; f_k\in L^p(\real, w).$$
   \end{enumerate}
\end{thm}

\begin{proof}
 The type $(p, p)$ inequality is proved in the same way as that of Theorem~\ref{Thm:variation-Ap+vector}. Now the extrapolation techniques can be found in sections~IV.5 and V.6 of \cite{gar-rubio} (see, in particular, Remark~V.6.5 there). We will only show the weak type $(1, 1)$ inequality. This proof is similar to the corresponding one of Theorem~\ref{vq CZ}, so we will indicate only the necessary modifications in the style of the proof of Theorem~\ref{Thm:variation-Ap+vector}.

Let $f:\real\to\el^\rho$ be a compactly supported integrable function, so $f=\{f_k\}$. Applying the Calder\'on-Zygmund decomposition  to $\l=1$ and the function $\f$ given by $\f(x)=\|f(x)\|_{\el^\rho}$, we get the corresponding decomposition of $f$: $f=g+b$. We maintain all notation introduced in the proof of Theorem~\ref{vq CZ}. The good part $g$ and the part of $b$ inside $\wt\O$ are handled as in the scalar valued case. We need  only to show the following inequality
  $$
 w\big(\big\{x\in\wt\O^c: \sum_k\big(\V_q \K^+(b_k)(x)\big)^\rho>\frac{1}{2^\rho}\big\}\big)
  \lesssim \frac1\l\int_{\real}\big(\sum_k|f_k(x)|^\rho\big)^{1/\rho}w(x)dx.
 $$

For every $k$ and $x\notin \wt\O$ choose an increasing sequence $\{t_{j, k}\}_j$ such that
 $$\V_{q}\K(b_k) (x)\le 2\big(\sum_j\big|K_{(t_{j, k},\,t_{j+1,k}]}(b_k)(x)\big|^q\big)^{1/q}.$$
Let $R_{j, k}=R_{(t_{j, k},\,t_{j+1,k}]}$ and $K_{R_{j,k}}=K_{(t_{j, k},\,t_{j+1,k}]}$. Accordingly, we introduce the index sets $\I^1_{j,k}$ and $\I^2_{j,k}$. Then
 \begin{align*}
 \Big(\sum_k\big(\V_{q}\K(b_k) (x)\big)^{\rho}\Big)^{1/\rho}
 &\le
  2\Big(\sum_k\big(\sum_{j}\big|\sum_{i\in\I^1_{j,k}}K_{R_{j,k}}(b_{i,k})(x)\big|^q\big)^{\rho/q}\Big)^{1/\rho}\\
   &+ 2\Big(\sum_k\big(\sum_{j}\big|\sum_{i\in\I^2_{j,k}}K_{R_{j,k}}(b_{i,k})(x)\big|^q\big)^{\rho/q}\Big)^{1/\rho}.
   \end{align*}
For the first term on the right we use the estimate already obtained before for each $b_{i,k}$:
 $$ \big(\sum_{j}\big|\sum_{i\in\I^1_{j,k}}K_{R_{j,k}}(b_{i,k})(x)\big|^q\big)^{1/q}
 \lesssim \sum_i\frac{|I_i|^\d}{|x-c_i|^{1+\d}}\int_{\real}|b_{i,k}(y)|dy.$$
Consequently,
  \begin{align*}
  \Big(\sum_k\big(\sum_{j}\big|\sum_{i\in\I^1_{j,k}}K_{R_{j,k}}(b_{i,k})(x)\big|^q\big)^{\rho/q}\Big)^{1/\rho}
  &\lesssim \sum_i\frac{|I_i|^\d}{|x-c_i|^{1+\d}}\int_{\real}\big(\sum_k|b_{i,k}(y)|^\rho\big)^{1/\rho}dy\\
  &=\sum_i\frac{|I_i|^\d}{|x-c_i|^{1+\d}}\int_{\real}\|b_{i}(y)\|_{\el^\rho}dy\\
  &\lesssim \sum_i\frac{|I_i|^\d}{|x-c_i|^{1+\d}}\int_{I_i}\|f(y)\|_{\el^\rho}dy.
  \end{align*}
Then as in the scalar case we deduce
 $$w(\{x \in \wt\O^c: \Big(\sum_k\big(\sum_{j}\big|\sum_{i\in\I^1_{j,k}}K_{R_{j,k}}(b_{i,k})(x)\big|^q\big)^{\rho/q}\Big)^{1/\rho}
 > \frac1{8} \}) \lesssim \int_{\real}\|f(y)\|_{\el^\rho}w(y)dy.$$

To handle the part on $\I^2_{j,k}$ choose $r$ such that $1<r\le\min(q,\,\rho)$. Then for each $k$ we have ($r'$ being the conjugate index of $r$)
 \begin{align*}
 \Big(\sum_{j}\big|\sum_{i\in\I^2_{j,k}}K_{R_{j,k}}(b_{i,k})(x)\big|^q\Big)^{1/q}
 &\le  \Big(\sum_{j}\big|\sum_{i\in\I^2_{j,k}}K_{R_{j,k}}(b_{i,k})(x)\big|^r\Big)^{1/r}\\
 &\le 4^{1/r'} \Big(\sum_i \big(\sum_{j}|K_{R_{j,k}}(b_{i,k})(x)|\big)^r\Big)^{1/r}.
 \end{align*}
However, for any $k$ and $x\notin \wt\O$
 \begin{align*}
 \sum_{j}|K_{R_{j,k}}(b_{i,k})(x)|
 \lesssim \frac1{|x-c_i|}\int_{I_i}|b_{i,k}(y)|dy.
 \end{align*}
Therefore,
 \begin{align*}
 \Big(\sum_k\big(\sum_{j}\big|\sum_{i\in\I^2_{j,k}}K_{R_{j,k}}(b_{i,k})(x)\big|^q\big)^{\rho/q}\Big)^{1/\rho}
 &\le 4^{1/r'} \Big(\sum_i \big( \frac1{|x-c_i|}\int_{I_i}\|b_{i}(y)\|_{\el^\rho}dy\big)^r\Big)^{1/r}\\
 &\le 4^{1/r'} \Big(\sum_i \frac{|I_i|^{r-1}}{|x-c_i|^r}\int_{I_i}\|f(y)\|_{\el^\rho}dy\Big)^{1/r}
  \end{align*}
 Thus
  \begin{align*}
   w(&\{x \in \wt\O^c: \Big(\sum_k\big(\sum_{j}\big|\sum_{i\in\I^2_{j,k}}K_{R_j}(b_{ik})(x)\big|^q\big)^{\rho/q}
   \Big)^{1/\rho}> \frac1{8} \}) \\
   & \lesssim \sum_i\int_{\wt\O^c}\frac{|I_i|^{r-1}}{|x-c_i|^r}\int_{I_i}\|f(y)\|_{\el^\rho}dy\,w(x)dx\\
   &\lesssim \int_{\real}\|f(y)\|_{\el^\rho}w(y)dy.
   \end{align*}
The proof of the theorem is complete.
 \end{proof}


\section{Application to Ergodic Theory}\label{Sec:ergodicTheory}


In this section we will give an application of Theorem~\ref{Thm:variation-Ap+} to ergodic theory. In the following $(X,\mathcal{F}, \mu)$ will always denote a $\s$-finite measure space.

Recall that a $\sigma$-endomorphism $\Phi$ of $(X,\mathcal{F}, \mu)$ is an endomorphism  of $\mathcal{F}$ modulo $\mu$-null sets as a Boolean $\sigma$-algebra. This means:
\begin{enumerate}[$\bullet$]
  \item $\Phi\left(\bigcup_{n=1}^{\infty}E_n\right)=\bigcup_{n=1}^{\infty}\Phi E_n$ for disjoint $E_n\in \mathcal{F}$;
  \item $\Phi(X\setminus E)=\Phi X\setminus\Phi E$ for all $E\in \mathcal{F}$;
  \item $E\in \mathcal{F},\mu(E)=0\implies\mu(\Phi E)=0$.
\end{enumerate}
Such a $\Phi$ induces a unique  linear operator, also denoted by $\Phi,$ on the space of measurable functions. The action of $\Phi$ on simple functions is given by
 $$\Phi(\sum_i a_i \un_{E_i}) = \sum_i a_i \un_{\Phi(E_i)}.$$
$\Phi$ preserves almost everywhere convergence. Namely, if $f_n\to f$ a.e., then $\Phi(f_n)\to \Phi(f)$ a.e. too.   $\Phi$ is clearly positive in the sense that $\Phi(f)\ge0$ whenever $f\ge0$. Moreover, $\Phi$ possesses the following properties:
 \begin{equation}\label{Pro:endomorphism}
\Phi(fg)=\Phi(f)\Phi(g), \quad |\Phi(f)|^p = \Phi(|f|^p).
 \end{equation}

\begin{rk}\label{extension}
  It is well-known and also  easy to check that $\Phi$ extends to the vector-valued setting. Let $B$ be a Banach space. Then $\Phi$ extends to the space of all strongly measurable functions from $X$ to $B$. For instance, the action of $\Phi$ on $B$-valued simple functions is again given as before:
 $$\Phi(\sum_i b_i \un_{E_i}) = \sum_i b_i\Phi(\un_{E_i}), \quad b_i \in B.$$
Then one can easily show that for any strongly measurable $B$-valued function $f$
  $$\|\Phi(f)(x)\|_B = \Phi(\|f\|_B)(x), \quad \forall\;x \in X.$$
 \end{rk}

In the sequel we will fix $1\le p < \infty$ and assume that $T$ is a bounded positive invertible operator on  $L^p(X,\mathcal{F}, \mu)$ with positive inverse. Then by \cite{K} there exists a unique $\sigma$-endomorphism $\Phi$ and a unique measurable positive function
 $h$ on $X$ such that $T(f) = h \Phi(f)$ for any $f\in L^p(X,\mathcal{F}, \mu)$.
 Moreover, $\Phi$ is a bijection on $\mathcal{F}$ modulo null sets, so $\Phi^{-1}$ is a $\sigma$-endomorphism too.

Note that for each $i\in\ent$  the operator $T^i$ has the same properties as $T$. The associated $\sigma$-endomorphism is $\Phi^i$. Let $h_i$ be the associated modular function. It is then easy to show that
  \beq\label{h}
  T^i(f)= h_i \Phi^i (f)  \quad \textrm{and}\quad h_{i+j} = h_i \Phi^i( h_j) , \quad i,j \in \mathbb{Z}.
  \eeq
On the other hand, for each $i$ the measure $\mu_i$, defined by $\mu_i(E)=\mu(\Phi^{i}(E))$,  is absolutely continuous with respect to $\mu$. Let $D_i=d\mu_i/d\mu$. Then
 $$\int_X\Phi^i(f)d\mu=\int_XfD_id\mu \quad \textrm{and}\quad D_{i+j}=D_i\Phi^{-i}(D_j), \quad i,j \in \mathbb{Z}.$$
Let $J_i=D_{-i}$. We then deduce
 \beq\label{J}
 \int_X J_i \Phi^i(f) d\mu = \int_X f d\mu \quad \textrm{and}\quad J_{i+j} = J_i \Phi^i (J_j ), \quad  i,j \in \mathbb{Z}.
  \eeq
 Now consider the sequence of the ergodic averages of $T$:
  $$\mathcal{A}^+(T)(f)(x) = \big\{ A^+_N(T)(f)(x) \big\}_{N\ge0}
  = \Big\{ \frac1{N+1} \sum_{n=0}^N T^n (f)(x) \Big\}_{N\ge0}\;, \quad x \in X.$$

\begin{prop}\label{Thm:main-weighted}
Let $1<p<\infty$ and  $T$ be as above. Then the following conditions are equivalent
 \begin{enumerate}[\rm (i)]
 \item There exists a constant $C>0$ such that
 $$\sup_N\int_X| A^+_N(T)(f)(x)|^p d \mu(x) \le C \int_X |f(x)|^p d \mu(x). $$

 \item  For any $2<q <\8$ $($or equivalently, for one $2<q <\8)$ there exists a constant $C> 0$ such that
 $$\int_X\big(\V_q\A^+(T)(f)(x)\big)^p d \mu(x) \le C \int_X |f(x)|^p d \mu(x). $$
\end{enumerate}
\end{prop}

Note that (i) simply means that $T$ is a mean bounded operator. It is obvious that (ii)$\implies$(i). The other implication improves a maximal ergodic theorem of \cite{MT1} where the $q$-variation in (ii) is replaced by the ergodic maximal function. It also extends the variation inequality of \cite{Bour} for $p=2$ and of \cite{LMX} for any $1<p<\8$ for positive contractions on $L^p(X,\mathcal{F}, \mu)$. Another advantage of the variation inequality in (ii) over the maximal ergodic inequality of \cite{MT1} is the corollary that (ii) immediately implies the pointwise ergodic theorem of \cite{MT1}.

 \begin{proof}
 It is known from \cite{MT1} that  (i) is equivalent to:
 \begin{enumerate}
 \item[(i$'$)] {\it  For almost all $x\in X$ the function defined on $\ent$ by $ i \mapsto w_x(i)=h_i(x)^{-p} J_i(x)$
is an $A_p^+$ weight with relevant constant independent of $x$.}
 \end{enumerate}
Therefore,  it remains to prove (i$'$)$\implies$(ii). To this end, it suffices  to show  that there exists a constant $C$ such that for any positive integer $N$
\beq\label{vM}
\int_X\norm{\mathcal{A}^+(T)(f)(x)}_{v_q(N)}^p d\mu(x)  \le C \int_X |f(x)|^p d \mu(x),
\eeq
where $\|a\|_{v_{q(N)}}$ is defined as in (\ref{variation}) with the supremum running over all increasing sequences of integers $0 \le N_0 \le N_1 \le \dots \le N.$

Let $L\in \mathbb{N}$ and $0\le i\le L$. By \eqref{J}, \eqref{Pro:endomorphism}, Remark \ref{extension} and \eqref{h}, we have
 \begin{align*}
 \int_X \big\|\A^+(T)(f)(x)\big\|_{v_q(N)}^p d\mu(x)
 &= \int_X   J_i(x) \big\|\Phi^i (\A^+(T)(f))(x) \big\|_{v_q(N)}^p  d\mu(x) \\
 &= \int_X  \big\|T^i (\A^+(T)(f))(x) \big\|_{v_q(N)}^p w_x(i) d\mu(x) \\
 &= \int_X  \big\|\A^+(T) (T^i (f))(x) \big\|_{v_q(N)}^p w_x(i) d\mu(x).
  \end{align*}
 Summing up over $i$ yields
 $$
 \int_X \big\|\A^+(T)(f)(x)\big\|_{v_q(N)}^p d\mu(x) =
 \frac1{L+1}\int_X \sum_{i=0}^L \big\|\A^+(T) (T^i (f))(x) \big\|_{v_q(N)}^p w_x(i) d\mu(x) .
 $$
Now define the function $g_x$ on $ \mathbb{Z} $ by $i \mapsto g_x(i)=\un_{[0,\, L+N]}(i) T^{i}(f)(x)$. Then
 $$A^+_k (g_x)(i)=\frac1{k+1} \sum_{n=0}^k \un_{[0,\,N+L]}(n+i) T^{n+i}(f)(x);$$
so
 $$\big\|\A^+(T) (T^i (f))(x) \big\|_{v_q(N)}=\big\|\A^+(g_x) \big\|_{v_q(N)}\le \big\|\A^+(g_x) \big\|_{v_q},
 \quad\forall\; x\in X.$$
Thus, applying Theorem~\ref{Thm:variation-Ap+} and (i$'$), we obtain
 $$ \sum_{i=0}^L \big\|\A^+(T) (T^i (f))(x) \big\|_{v_q(N)}^p w_x(i)
  \lesssim \sum_i  \un_{[0,L+N]}(i)|T^{i} (f)(x) |^p  w_x(i) $$
 for almost all $x\in X$.
 Taking integral over $X$ implies
   \begin{align*}
   \int_X\sum_{i=0}^L \big\|\A^+(T) (T^i (f))(x) \big\|_{v_q(N)}^p w_x(i) d\mu(x)
  &\lesssim \sum_{i=0}^{L+N}\int_X|T^{i} (f)(x) |^p  w_x(i)  d\mu(x)\\
  &=(L+N+1)\|f\|_p^p.
  \end{align*}
Therefore, we deduce
 $$\int_X \big\|\A^+(T)f(x)\big\|_{v_q(N)}^p d\mu(x)
\lesssim \frac{ L+N+1}{L+1} \|f\|_p^p.$$
Then letting $L\to\8$ yields \eqref{vM}.
\end{proof}

It is natural to wonder whether Proposition~\ref{Thm:main-weighted} has a weak type $(1, 1)$ substitute  for $p=1$. We do not know how to solve this problem in the general case. However, we have the following partial result.

\begin{prop}\label{weighted1}
  Keep the previous assumption  on $T$ with $p=1$. Assume in addition that  $h=1$, i.e., $T=\Phi$.
 If $T$ is mean bounded on $L^1(X,\mathcal{F},\mu)$:
 $$\sup_N\int_X| A^+_N(T) (f)(x)| d \mu(x) \le C \int_X |f(x)| d \mu(x),$$
then for any $2<q<\8$
 $$\mu\big(\big\{x:\,  \mathcal{V}_q\A^+(T)(f)(x) > \lambda\big\}\big) \lesssim \frac{1}{\lambda} \int_X |f(x)| d \mu(x),\quad \forall\; \l>0,\;
 f\in L^1(X,\mathcal{F},\mu). $$
\end{prop}

\begin{proof}
 The mean boundedness of $T$ and \eqref{J} imply  that for all $N\in \mathbb{N}$ and $i \in \mathbb{Z}$
 $$\int_X \frac1{N+1} \sum_{n=0}^N \Phi^n(f) (x) d\mu(x) \lesssim\int_X J_i(x) \Phi^i( f)(x) d\mu(x), \quad\forall\; f\ge0.$$
Note that $J_i^{-1} = \Phi^i(J_{-i})$. Hence, if we call $g = J_i\Phi^i( f)$, that is $f = \Phi^{-i} ( \Phi^i(J_{-i}) g) = J_{-i}\Phi^{-i}(g)$,   by the last inequality and \eqref{J} we deduce that
 \begin{align*}
 \int_X g(x) d\mu(x)
 &\gtrsim\frac1{N+1} \sum_{n=0}^N \int_X \Phi^n(J_{-i}\Phi^{-i}(g)) (x) d\mu(x) \\
 &=\frac1{N+1} \sum_{n=0}^N \int_X \Phi^{n-i}(\Phi^i(J_{-i}))(x) \Phi^{n-i}(g) (x) d\mu(x)\\
&=\frac1{N+1} \sum_{n=0}^N \int_X \Phi^{n-i}(J_i^{-1}g) (x) d\mu(x)\\
&=\frac1{N+1} \sum_{n=0}^N \int_X J_{n-i}(x) \Phi^{n-i}(J_{i-n})(x) \Phi^{n-i}(J_i^{-1}g) (x) d\mu(x)\\
&=\frac1{N+1} \sum_{n=0}^N \int_X J_{i-n}(x) J_i^{-1}(x)g (x) d\mu(x).
\end{align*}
Thus the function
$x \mapsto \frac1{N+1} \sum_{n=0}^N J_{i-n}(x) J_i^{-1}(x)$ has to be in $L^\infty(X,\mathcal{F},\mu)$ and its $L^\infty$-norm is majorized by a constant $C$ independent of  $N$ and $i$:
 $$\frac1{N+1} \sum_{n=0}^N J_{i-n}(x) \le C J_i(x). $$
That is,  for almost all  $x$ the function $ i \mapsto  J_i(x)$ on  $\mathbb{Z}$  is an $ A_1^+$ weight with constant independent of  $x$.

Now we are ready to prove the announced weak type $(1, 1)$ inequality by transference.  As in the proof of the previous proposition, it suffices to consider $\|\mathcal{A}^+(T)(f)(\Phi^{i}(x))\|_{v_q(N)}$ for any positive integer $N$.
Given $\lambda >0$ let  $E_\lambda = \{x: \|\mathcal{A}^+(T)(f)(x)\|_{v_q(N)} > \lambda\}.$ Let $i \in \mathbb{Z}$. By \eqref{J} we have
 \begin{align*}
 \mu(E_\lambda)
 &= \int_X J_i(x) \Phi^i(\chi_{E_\lambda})(x) d\mu(x)\\
 &= \int_{\big\{x: \|\Phi^i(\mathcal{A}^+(T)f)(x)\|_{v_q(N)} > \lambda\big\}} J_i(x) d\mu(x)\\
 &= \int_{\big\{x: \|\mathcal{A}^+(\Phi)\Phi^i(f))(x)\|_{v_q(N)} > \lambda\big\}} J_i(x) d\mu(x).
\end{align*}
It then follows that
 \begin{align*}
 \mu(E_\lambda)
 &=\frac1{L+1} \sum_{i=0}^L\int_{\big\{x: \|\mathcal{A}^+(\Phi)\Phi^i(f))(x)\|_{v_q(N)} > \lambda\big\}} J_i(x) d\mu(x) \\
 &=\frac1{L+1}\int_X  \sum_{i=0}^L\un_{F_\l}(x, i) J_i(x) d\mu(x),
 \end{align*}
where
 $$F_\l=\big\{(x,i): \big\| \big\{\frac1{N+1}\sum_{n=0}^N \Phi^{n+i}(f)(x)\un_{[0,\,L+N]}(i+n) \big\}_N\big\|_{v_q(N)}
 > \lambda\big\}\subset X\times\ent.$$
Since the function $ i \mapsto  J_i(x)$ on  $\mathbb{Z}$  is an $ A_1^+$  weight with constant independent of  $x$,  applying Theorem \ref{Thm:variation-Ap+} to the function $i \mapsto \un_{[0, \,L+N]}(i) \Phi^{i}(f)(x)$, we get
 \begin{align*}
 \mu(E_\lambda)
 &\lesssim\frac1{\l(L+1)}\int_X  \sum_{i=0}^\infty \un_{[0, \,L+N]}(i) |\Phi^{i}(f)(x)| J_i(x) d\mu(x)\\
 &=\frac{L+N+1}{\l(L+1)}  \int_X  |f(x)|  d\mu(x).
 \end{align*}
 This implies the desired weak type $(1, 1)$ inequality.
 \end{proof}

The proposition above immediately gives the following pointwise ergodic theorem.

\begin{cor}
 Under the assumption of Proposition~\ref{weighted1}, $A^+_N(f)$ converges almost everywhere as $N\to\8$ for any $f\in L^1(X,\mathcal{F},\mu)$.
 \end{cor}

\begin{rk}
  Mart\'{\i}n-Reyes and de la Torre also considered symmetric ergodic averages in \cite{MT1}. Let
 $$A_N(T) (f)  =  \frac1{2N+1} \sum_{n=-N}^N T^n (f).$$
Under the assumption of Proposition~\ref{Thm:main-weighted}, they showed the following equivalence
 \begin{enumerate}[(i)]
 \item There exists a constant $C$ such that
 $$\sup_N \|A_N(T)(f)\|_p\le C \|f\|_p,\quad\forall\; f\in  L^p(X,\mathcal{F},\mu).$$
 \item There exists a constant $C'$ such that
  $$\int_X\sup_N  |A_N(T)(f)|^pd\mu\le C'\int_X|f|^pd\mu,\quad\forall\; f\in  L^p(X,\mathcal{F},\mu).$$
  \item For almost all $x$ the function  $ i \mapsto h_i(x)^{-p} J_i(x)$ is an $A_p$ weight with relevant constant independent of $x$.
 \end{enumerate}
Using Corollary~\ref{Thm:variation-Ap}, we can show the corresponding symmetric analogue of Proposition~\ref{Thm:main-weighted}. Equally,  Proposition~\ref{weighted1} admits a symmetric version.
\end{rk}

\begin{rk}
 Using Theorem~\ref{Thm:variation-Ap+vector} instead of Theorem~\ref{Thm:variation-Ap+}, one can show that the maximal ergodic inequalities in this section extend to the vector-valued case too. We leave the details to the interested reader.
\end{rk}

\noindent{\bf Acknowledgments.} We thank Wei Chen for pointing out to us the reference \cite{HLP} after he has seen our manuscript in {\it arXiv}. We also acknowledge the financial supports of NSFC grant No. 11271292, MTM2011-28149-C02-01 and ANR-2011-BS01-008-01.


\end{document}